\documentclass[12pt,a4paper]{amsart}
\usepackage{amssymb,,verbatim,ulem,soul}
\usepackage{graphicx,amssymb,amsfonts,epsfig,amsthm,a4,amsmath,url,graphicx,enumerate}
\usepackage[latin1]{inputenc}
\usepackage[all]{xy}

\usepackage[usenames,dvipsnames]{color}
\newtheorem*{thmm}{Theorem}

\theoremstyle{definition}

\numberwithin{equation}{section}
\newtheorem{thm}[equation]{Theorem}
\newtheorem{cor}[equation]{Corollary}
\newtheorem{lem}[equation]{Lemma}

\newtheorem{prop}[equation]{Proposition}
\theoremstyle{definition}
\newtheorem{defn}[equation]{Definition}

\newtheorem{que}[equation]{Question}

\newtheorem{rem}[equation]{Remark}
\newtheorem{ex}[equation]{Example}

\newtheorem{ThmIntro}{Theorem}
\newtheorem{CorIntro}[ThmIntro]{Corollary}
\newtheorem{PropIntro}[ThmIntro]{Proposition}
\theoremstyle{definition}
\newtheorem{DefnIntro}[ThmIntro]{Definition}
\newtheorem{RemIntro}[ThmIntro]{Remark}

\newcommand{\eps}{\varepsilon}
\newcommand{\Z}{\mathbf{Z}}
\newcommand{\F}{\mathbf{F}}
\newcommand{\N}{\mathbf{N}}
\newcommand{\R}{\mathbf{R}}
\newcommand{\Q}{\mathbf{Q}}

\newcommand{\K}{\mathbf{K}}
\newcommand{\g}{\mathfrak{g}}

\newcommand{\mk}{\mathfrak}

\newcommand{\Hom}{\textnormal{Hom}}
\newcommand{\supp}{\textnormal{supp}}
\newcommand{\Aut}{\textnormal{Aut}}

\newcommand{\GL}{\textnormal{GL}}

\newcommand{\SL}{\textnormal{SL}}

\newcommand{\SOL}{\textnormal{SOL}}

\newcommand{\Aff}{\textnormal{Aff}}

\newcommand{\Contr}{\textnormal{Contr}}
\newcommand{\Ind}{\textnormal{Ind}}

\newcommand{\End}{\textnormal{End}}

\newcommand{\Diam}{\textnormal{Diam}}

\newcommand{\lp}{(\!(}
\newcommand{\rp}{)\!)}

\newcommand{\HFD}{\mathcal{H}_{\mathrm{fd}}}
\newcommand{\HT}{\mathcal{H}_{\mathrm{t}}}

\newcommand{\WAPT}{\mathcal{WAP}_{\mathrm{t}}}
\newcommand{\WAPAP}{\mathcal{WAP}_{\mathrm{ap}}}
\newcommand{\WAPFD}{\mathcal{WAP}_{\mathrm{fd}}}
\newcommand{\APFD}{\mathcal{AP}_{\mathrm{fd}}}
\newcommand{\APT}{\mathcal{AP}_{\mathrm{fd}}}
\newcommand{\KK}{\mathcal{K}}
\newcommand{\KT}{\mathcal{K}^\dagger}

\begin{document}

\title[Vanishing of reduced cohomology for Banach representations]{On the vanishing of reduced 1-cohomology for Banach representations}


\author[Cornulier]{Yves Cornulier}
\address{CNRS and Univ Lyon, Univ Claude Bernard Lyon 1, Institut Camille Jordan, 43 blvd. du 11 novembre 1918, F-69622 Villeurbanne}
\email{cornulier@math.univ-lyon1.fr}
\author[Tessera]{Romain Tessera}
\address{Laboratoire de Math\'ematiques\\
B\^atiment 425, Universit\'e Paris-Sud 11\\
91405 Orsay\\FRANCE}
\email{romain.tessera@math.u-psud.fr}



		
\date{November 15, 2017}

\begin{abstract}
A theorem of Delorme states that every unitary representation of a connected Lie group with nontrivial reduced first cohomology has a finite-dimensional subrepresentation. More recently Shalom showed that such a property is inherited by cocompact lattices and stable under coarse equivalence among amenable countable discrete groups. We give a new geometric proof of Delorme's theorem which extends to a larger class of groups, including solvable $p$-adic algebraic groups, and finitely generated solvable groups with finite Pr\"ufer rank.  
 Moreover all our results apply to isometric representations in a large class of Banach spaces, including reflexive Banach spaces. As applications, we obtain an ergodic theorem in for integrable cocycles, as well as a new proof of Bourgain's Theorem that the 3-regular tree does not embed quasi-isometrically into any superreflexive Banach space.
\end{abstract} 

\maketitle

\setcounter{tocdepth}{1}
\tableofcontents
\section{Introduction}
\subsection{Background}

Let $G$ be a locally compact group. We consider representations of $G$ into Banach spaces. It is thus convenient to call $G$-module a Banach space $V$ endowed with a representation $\rho$ of $G$, by bounded automorphisms, in a way that the mapping $g\mapsto gv=\rho(g)v$ is continuous for all $v\in V$. We will denote by $V^G$ the subspace of $G$-fixed points.

The space $Z^1(G,V)$ (also denoted $Z^1(G,\rho)$) of 1-cocycles is the set of continuous maps $b:G\to V$ satisfying the 1-cocycle condition $\rho(gh)=\rho(g)b(h)+b(g)$. It is endowed with the topology of uniform convergence on compact subsets. The subspace of coboundaries $B^1(G,V)$ consists of those $b$ of the form $b(g)=v-\rho(g)v$ for some $v\in V$. It is not always closed, and the quotient of $Z^1(G,V)$ by its closure is called the first reduced cohomology space $\overline{H^1}(G,V)$ (or $\overline{H^1}(G,\rho)$). See notably the reference book \cite{Guit}.

Vanishing properties of the first reduced cohomology has especially been studied in the context of unitary representations on Hilbert spaces. If $G$ satisfies Kazhdan's Property T, it is a classical result of Delorme that $\overline{H^1}(G,V)=0$ (and actually $H^1$ itself vanishes) for every unitary Hilbert $G$-module $V$. See also \cite[Chapter 2]{BHV08}. For $G$ a discrete finitely generated group, the converse was established by Mok and Korevaar-Schoen \cite{Mok,KS}: if $G$ fails to satisfy Kazhdan's Property T then $\overline{H^1}(G,V)\neq 0$ for some unitary Hilbert $G$-module $V$. A more metrical proof was provided by Gromov in \cite{GrRW}, and Shalom \cite{Shal} extended the result to the setting of non-discrete groups, see also \cite[Chap.\ 3]{BHV08}. 

The groups in which we will be interested will usually be amenable (and non-compact), so that this non-vanishing result holds. However, it often happens that unitary representations with non-vanishing 1-cohomology for a given group are rare. For instance, it is an easy observation of Guichardet \cite{G72} that if $G$ is abelian, or more generally nilpotent, and $V$ is a Hilbert $G$-module with $V^G=0$, then $\overline{H^1}(G,\pi)=0$. In particular, the only irreducible unitary representation with non-vanishing $\overline{H^1}$ is the trivial 1-dimensional representation. Shalom \cite{Shalom} thus introduced the following terminology: if $G$ satisfies the latter property, it is said to satisfy Property $\HT$. He also introduced a natural slightly weaker invariant: $G$ has Property $\HFD$ if for every unitary Hilbert $G$-module with no $G$-submodule (= $G$-invariant closed subspace) of positive finite dimension, we have $\overline{H^1}(G,V)=0$. This can also be characterized as follows: $G$ has Property $\HFD$ if and only if every irreducible unitary representation with $\overline{H^1}\neq 0$ has finite dimension, and there are only countably many up to equivalence. Nontrivial examples were provided by the following theorem of P.\ Delorme. 

\begin{thmm}[\cite{Del}, Th.\ V6, Cor.\ V2]
Let $G$ be a connected solvable Lie group and let $V$ be an irreducible unitary $G$-module. Assume that $V$ is not a tensor power of a character occurring as quotient of the adjoint representation. Then $V$ has zero first reduced cohomology. 
\end{thmm}

Using Shalom's subsequent terminology, it follows that connected solvable Lie group have Property $\HFD$. Florian Martin \cite{FM} extended this (and the above theorem) to arbitrary amenable connected Lie groups.

Delorme's proof takes more than 10 pages, involving a lot of ad-hoc analytical arguments and strongly relies on representation theory of the Lie algebra. It implies in particular that such $G$ has Property $\HFD$. Shalom proved in \cite{Shalom} that Property $\HFD$ is invariant under passing to cocompact lattices. As a consequence, it is satisfied by virtually polycyclic groups: indeed such a group $\Gamma$ has a finite index subgroup $\Gamma'$ embedding as a cocompact lattice in a group $G$ as in the theorem, and then Property $\HFD$ thus successively passes from $G$ to $\Gamma'$ and then to $\Gamma$.

A motivation for Property $\HFD$ is to find interesting finite-dimensional representations. More precisely, if $G$ is an infinite discrete amenable group with Property $\HFD$, it is easy to deduce that $G$ admits an infinite virtually abelian quotient.

In this work, we provide a new, simpler proof of Delorme's theorem based on geometric/dynamical considerations. This allows to extend the previous results in two directions: first our approach allows to encompass a much larger class of groups, and second it allows to generalize it to uniformly bounded representations in more general Banach spaces.

A crucial feature which is used all the time in the context of unitary representations is the notion of ``orthogonal complement". Although this notion does not survive in the more general framework of Banach spaces, a weaker and yet very powerful property holds for a large class of Banach $G$-modules: the subspace of invariant vectors is a factor, in fact it admits a canonical complement. 

Recall that 
a $G$-module  $(V,\pi)$ is called weakly almost periodic (WAP for short) if for every $v\in V$, the orbit $\pi(G)v$ is relatively compact in $V$ in the weak topology. 
Note that this does not depend on a choice of topology on $G$. As observed in \cite{BRS}, it implies that $\rho$ is a uniformly bounded representation: $\sup_{g\in G}\|\rho(g)\|<\infty$, and in case $V$ is reflexive, this is equivalent to being a uniformly bounded representation. WAP representations turn out to be a convenient wide generalization of unitary representations.

\begin{DefnIntro}
We say that a locally compact group $G$ has Property $\WAPT$ if every WAP $G$-module $V$ with $V^G=\{0\}$ has $\overline{H^1}(G,V)=0$.
\end{DefnIntro}

Property $\WAPT$ is a strengthening of Property $\HT$. Using that there a $G$-invariant complement (\cite[Theorem 14]{BRS}, see \S\ref{sec:decoca}), this means that for every WAP $G$-module $V$, the reduced 1-cohomology is ``concentrated" in $V^G$. Observe that a 1-cocycle valued in $V^G$ is just a continuous group homomorphism. It follows, for instance that for a compactly generated, locally compact group $G$ without Kazhdan's Property T (e.g., amenable and non-compact), the condition $\mathrm{Hom}(G,\R)=0$ is an obstruction to Property $\HT$, and hence to Property $\WAPT$.  

This explains why it is, in the context of unitary representations, natural to deal with the more flexible Property $\HFD$. Defining an analogue in this broader context leads to some technical difficulties, which leads us to introduce two distinct notions. First, recall that a Banach $G$-module is almost periodic if all orbits closures are compact (in the norm topology). This is obviously a strengthening of being WAP. This is satisfied by finite-dimensional uniformly bounded modules, and it can be checked (Corollary \ref{sumfin}) that a uniformly bounded Banach $G$-module is almost periodic if and only if the union of its finite-dimensional submodules is dense. In general, the set of vectors whose $G$-orbit has compact closure is a closed submodule, denoted $V^{G,\mathrm{ap}}$, and, by the above, equals the closure of the union of finite-dimensional submodules.

\begin{DefnIntro}
Let $G$ be a locally compact group.
\begin{itemize}
\item $G$ has Property $\WAPAP$ if for every WAP Banach $G$-module $V$ with $V^{G,\mathrm{ap}}=0$ we have $\overline{H^1}(G,V)=0$;
\item $G$ has Property $\WAPFD$ if for every WAP Banach $G$-module $V$ and every 1-cocycle $b$ that is nonzero in $\overline{H^1}(G,V)$, there exists a closed submodule $W$ of nonzero finite codimension such that the projection of $b$ on $V/W$ is unbounded.
\end{itemize}
\end{DefnIntro}

Since finite-codimensional submodules are complemented (Proposition \ref{fcodim:compl}), Property $\WAPFD$ implies Property $\WAPAP$. We do not know if the converse holds (see Question \ref{qwaeq} and the discussion around it).

\subsection{Main results}
In the sequel, we abbreviate ``compactly generated locally compact" as ``CGLC".

Our main result is the fact that a relatively large class of CGLC groups, including connected solvable Lie groups, algebraic solvable $p$-adic groups, and finitely generated solvable groups with finite Pr\" ufer rank, satisfy property $\WAPAP$.  
Even in the case of connected solvable Lie groups, the proof is not merely an adaptation of Delorme's proof, which is specific to the Hilbert setting. Instead it uses a dynamical phenomenon which is very specific to these groups. 

In order to illustrate this dynamical phenomenon, let us examine the simplest example where it arises: the affine group $\Aff(\R):=U\rtimes A$, where $U\simeq A\simeq \R$, and where the group law is given by $(x,t)(y,s)=(x+e^ty,t+s)$. The important feature of this group is the fact that the normal subgroup $U$ is ``contracted" by the action of $A$: i.e.\ given $a_t=(0,t)\in A$ and $u=(x,0)\in U$, one has $$a_t^{-1}ua_t=(e^{-t}x,0),$$ from which we deduce that $a_t^{-1}ua_t\to (0,0)$ as $t\to \infty$. The group $\Aff(\R)$ turns out to have Property $\WAPT$, and, roughly speaking, the proof consists in proving that $U$ behaves as if it did not exist at all, so that everything boils down to the fact that $A$ itself satisfies $\WAPT$. 

An elaboration of this argument applies to the following ad-hoc class of groups:

\begin{DefnIntro}\label{defc}
Denote by $\mathfrak{C}$ the class of (solvable) locally compact groups $G$ having
 two closed subgroups $U$ and $N$ such that
\begin{enumerate}
\item $U$ is normal and $G=UN$;
\item\label{ssg0} $N$ is a CGLC compact-by-nilpotent group, i.e., is compactly generated and has a compact normal subgroup such that the quotient is nilpotent;
\item\label{ssg1} $U$ decomposes as a finite direct product $\prod U_i$, where each $U_i$ is normalized by the action of $N$ and is an open subgroup of a unipotent group $\mathbb{U}_i(\K_i)$ over some non-discrete locally compact field of characteristic zero $\K_i$. 
\item\label{ssg3} $U$ admits a cocompact subgroup $V$ with, for some $k$, a decomposition $V=V_1V_2\ldots V_k$ where each $V_i$ is a subset such that there is an element $t=t_i\in N$ such that $t^{-n}vt^n\to 1$ as $n\to \infty$ for all $v\in V_i$.
\end{enumerate}
\end{DefnIntro}

This notably includes (see Proposition \ref{ex_clasc})
\begin{itemize}
\item real triangulable connected Lie groups;
\item groups of the form $G=\mathbb{G}(\mathbf{Q}_p)$, where $\mathbb{G}$ is a solvable connected linear algebraic group defined over the $p$-adic field $\Q_p$ such that $G$ is compactly generated;
\item mixtures of the latter, such as the semidirect product $(\mathbf{K}_1\times\mathbf{K}_2)\rtimes_{(t_1,t_2)}\Z$, where $\mathbf{K}_i$ is a nondiscrete locally compact field and $|t_i|\neq 1$.
\end{itemize}

Let us also pinpoint that in many cases, the method applies without the characteristic zero assumption in (\ref{ssg1}). Namely, assuming that $U=U_0\times\bigoplus_p U_p$ where $U_0$ is the characteristic zero part and $U_p$ is the $p$-torsion ($p$ ranges over primes, with only finitely many $p$ for which $U_p\neq 1$), this applies if for every $p>0$, $U_p$ is $(p-1)$-step-nilpotent, so that it naturally has a Lie algebra, and this Lie algebra has a $G$-invariant structure of Lie algebra over $\mathbf{F}_p(\!(t)\!)$.

Here is our first main result

\begin{ThmIntro}\label{Main}
Locally compact groups in the class $\mathfrak{C}$ have Property $\WAPT$. In particular, they have Property $\HT$.
\end{ThmIntro}

The proof of Theorem \ref{Main} involves several steps of independent interest, including the existence of ``strong controlled F\o lner subsets" for groups in the class $\mathfrak{C}$ (Theorem \ref{th:C_strong}).

In order to extend Theorem \ref{Main}, we use induction methods as in \cite{Shalom,BFGM} to obtain

\begin{ThmIntro}(see Section \ref{propwa})~\label{introstab}
\begin{enumerate}
\item Properties $\WAPT$ and $\WAPAP$ are inherited by closed cocompact subgroups such that the quotient has an invariant probability measure.
In particular these properties are inherited by closed cocompact subgroups among amenable groups.
\item Properties $\WAPT$ and $\WAPAP$ are invariant (i.e., both the property and its negation is stable) under taking quotients by compact normal subgroups.
\item Let $G$ be a locally compact group with a closed normal cocompact subgroup $N$. If $N$ has Property $\WAPT$ then $G$ has Property $\WAPAP$.
\item Let $\Lambda$ be a countable discrete group with Property $\WAPAP$ and $\Gamma$ a countable discrete group admitting an RCE (random cocompact embedding, see \S\ref{mecmec}) into $\Lambda$. Then $\Gamma$ also has Property $\WAPAP$.  
In particular, Property $\WAPAP$ is stable under coarse equivalence among countable discrete amenable groups. 
\end{enumerate}
\end{ThmIntro}
We introduce the terminology RCE to name a slight reinforcement of uniform measure equivalence used in \cite{Shalom}.  

\begin{DefnIntro}\label{defcp}
Let $\mathfrak{C}'$ be the larger class consisting of those locally compact groups $G$ such that there exists a sequence of copci (= continuous, proper with cocompact image) homomorphisms $G\to G_1\leftarrow G_2\to G_3$ such that the image of $G_2\to G_1$ is normal in $G_1$ and $G_3$ belongs to the class $\mathfrak{C}$. 
\end{DefnIntro}

This may sound a bit artificial, but the point is that this definition ensures that all amenable, virtually connected Lie groups belong to the class $\mathfrak{C}'$ (Proposition \ref{vcalfd}), as well as all groups with a open finite index subgroup in the class $\mathfrak{C}$. 

\begin{CorIntro}\label{CPWAPAP}
Locally compact groups in the class $\mathfrak{C}'$ have Property $\WAPFD$. In particular, they have Property $\HFD$.
\end{CorIntro}

\begin{CorIntro}\label{vcalfd2}
Every virtually connected amenable Lie group $G$ has Property $\WAPFD$.
\end{CorIntro}

In view of Proposition \ref{ex_clasc}, we deduce

\begin{CorIntro}\label{cor:realtriangulated}
Real-triangulable Lie groups and compactly generated amenable Zariski-(closed connected) subgroups of $\mathrm{GL}_n(\Q_p)$ have $\WAPT$.
\end{CorIntro}

We combine these results to obtain the following Banach space version of Delorme's theorem. 
 
\begin{CorIntro}\label{Main'} (see Corollary \ref{cor:Delorme})
Let $G$ be a connected solvable Lie group. Then every WAP $G$-module with nonzero first cohomology has a 1-dimensional factor (with nonzero first cohomology).
\end{CorIntro}

We can also apply Theorem \ref{Main} to deduce many new examples of discrete groups with Property $\WAPAP$.

\begin{CorIntro}
The class of groups satisfying Property $\WAPAP$ includes all discrete groups that are virtually cocompact lattices in a finite direct product of connected Lie groups and algebraic groups over $\Q_p$ (for various primes $p$). This includes polycyclic groups and more generally all amenable groups already known to satisfy Property $\HFD$ (see \cite{Shalom}).   
Some of these, being cocompact lattices in groups in the class $\mathfrak{C}$, actually  have Property $\WAPT$: this includes for instance of $\SOL$,  solvable Baumslag-Solitar groups and lamplighter groups $(\Z/n\Z)\wr \Z$.
\end{CorIntro}

We can also, along with additional structural work, deduce the following result, which answers a question of Shalom \cite{Shalom} (who asked whether these groups have $\HFD$).

Recall that a group has finite Pr\"ufer rank if for some $k$, all its finitely generated subgroups admit a generating $k$-tuple. Let us abbreviate ``virtually solvable of finite Pr\"ufer rank" to ``VSP".

\begin{ThmIntro}\label{VSPWAPAP} (Corollary \ref{cor:VSP})
Every finitely generated, VSP group has Property $\WAPFD$, and in particular has Property $\HFD$.
 \end{ThmIntro}

Finitely generated amenable (or equivalently, virtually solvable) subgroups of $\GL(d,\Q)$ are notably covered by the theorem: more precisely, these are precisely (when $d$ is allowed to vary) the virtually torsion-free groups in the class of finitely generated VSP groups. Actually, the theorem precisely consists of first proving it in this case, and deduce the general case using a recent result of Kropholler and Lorensen \cite{KL17}: every finitely generated VSP group is quotient of a virtually torsion-free finitely generated VSP group. The case of finitely generated amenable (or equivalently, virtually solvable) subgroups of $\GL(d,\Q)$ is based on an embedding result (see Proposition \ref{propgp} and \ref{ccpp}) of independent interest.

\begin{RemIntro}For a solvable group $G$ with derived series $G_0=G$, $G_{i+1}=[G_i,G_i]$, the sum $\sum_{i\ge 0}\dim_\Q((G_i/G_{i+1})\otimes\Q)$ is known as the Hirsch number of $G$. 

Shalom conjectured \cite[\S 6.6]{Shalom} that a finitely generated solvable group with finite Hirsch number has Property $\HFD$. This was recently disproved by Brieussel and Zheng, who showed that the groups $(\Z/m\Z)\wr \Z^d$ for $d\geq 3$ and $m\geq 2$ do not have Property $\HFD$. (Shalom previously proved that $\Z\wr\Z$ does not have Property $\HFD$: all this shows that our methods cannot apply to such ``large" solvable groups.)

However, it follows from Theorem \ref{VSPWAPAP} that the conjecture holds in the case of torsion-free groups: indeed, it is known \cite{K86} that finitely generated torsion-free solvable groups with finite Hirsch number are VSD groups.
\end{RemIntro}




We deduce the following strengthening of \cite[Theorem 1.3]{Shalom}, which is the particular case of polycyclic groups.

\begin{CorIntro}
Let $\Lambda$ be a finitely generated, (virtually) solvable group of finite Pr\"ufer rank. Let $\Gamma$ be a finitely generated group quasi-isometric to $\Lambda$. Then $\Gamma$ has a finite index subgroup with infinite abelianization. 
\end{CorIntro}
\begin{proof}
This consists in combining Theorem \ref{VSPWAPAP} with two results of Shalom:
\begin{itemize}
\item \cite[Theorem 4.3.1]{Shalom}, which says that every infinite finitely generated amenable group with Property $\HFD$ has a finite index subgroup with infinite abelianization;
\item \cite[Theorem 4.3.3]{Shalom}: among finitely generated amenable groups, Property $\HFD$ is a quasi-isometry invariant.\qedhere
\end{itemize}
\end{proof}

\subsection{Cocompact hull of amenable subgroups of $\GL(d,\Q)$}\label{ichas}
Our proof of Theorem \ref{VSPWAPAP} relies on a construction of independent interest. 
We start introducing a second variant of the class $\mathfrak{C}$.
\begin{DefnIntro}\label{dcpp} Let $\mathfrak{C}''$ be the class of compactly generated locally compact groups defined as the class $\mathfrak{C}$ (Definition \ref{defc}) but replacing (\ref{ssg1}) with: $N$ has polynomial growth. \end{DefnIntro}

\begin{ThmIntro}\label{thm:ccpp}
Every finitely generated amenable (= virtually solvable) subgroup of $\GL_m(\Q)$ embeds as a compact lattice into a locally compact group $G$ with an open subgroup of finite index $G'$ in the class $\mathfrak{C}''$.
\end{ThmIntro}

This is a key step in the proof of Theorem \ref{VSPWAPAP}. Let us pinpoint other consequences of the existence of this cocompact hull.

Let $G$ be a compactly generated locally compact group, and let $\lambda$ be the left representation of $G$ on real-valued functions
on $G$, namely $\lambda(g)f(x)=f(g^{-1}x)$. We let $S$ be a compact symmetric generating subset of $G$. For any $1\leq p\leq
\infty$, and any subset $A$ of $G$, define
$$J_p(A)=\sup_{f}\frac{\|f\|_p}{\sup_{s\in S}\|f-\lambda(s)f\|_p},$$
where $f$ runs over functions in $L^p(G)$, supported in $A$. Recall \cite{tes3} that the $L^p$-isoperimetric
profile inside balls is given by
$$J^b_{G,p}(n)=J_p(B(1,n)).$$
\begin{CorIntro} (see Corollary \ref{cor:Jb})\label{corjbi}
For every finitely generated VSP group $G$ equipped with a finite generating subset $S$, we have 
\[J^b_{G,p}(n)\succeq n,\]
i.e., there exists $c>0$ such that $J^b_{G,p}(n)\geq cn$ for all $n$.
\end{CorIntro}

Corollary \ref{corjbi} has a consequence in terms of equivariant $L^p$-compression rate.
Recall that the equivariant $L^p$-compression rate $B_p(G)$ of a
locally compact compactly generated group is the supremum of those $0\leq \alpha\leq 1$ such that there
exists a proper isometric affine action $\sigma$ on some $L^p$-space satisfying, for all $g \in G$,
$\|\sigma(g).0\|_p \geq  |g|_S^{\alpha}-C$
for some constant $C<\infty$. It follows from \cite[Corollary 13]{tes1} that for a group $G$ with $J^b_{G,p}(n)\succeq n$, we have  
$B_p(G)=1$; hence 
\begin{CorIntro}
Let $1\leq p<\infty$, and $G$ be a finitely generated VSP group. Then $B_p(G)=1$.
\end{CorIntro}

See also \cite[Theorem 10]{tes1} for a finer consequence.

\subsection{An mean ergodic theorem in $L^1$}
Let $X$ be a probability space and let $T:X\to X$ be a measure-preserving ergodic self-map of $X$.
Recall that Birkhoff's theorem states that for all $f\in L^1(X)$, the sequence 
$\frac{1}{n}\sum_{i=0}^{n-1}T^i(f)$ converges a.e.\ and in $L^1$ to the integral of $f$. 
Observe that the map $n\mapsto c(n)=\sum_{i=0}^{n-1}T^i(f)\in L^1(X)$ (a priori well defined on positive integers, and more generally on $\Z$ if $T$ is invertible) satisfies the cocycle relation: $c(n+1)=T(c(n))+c(1)$. Hence, assuming that $T$ is invertible, Birkhoff's ergodic theorem can be restated in a more group-theoretic fashion: given an ergodic measure measure-preserving action of $\Z$ on a probability space $X$, every continuous cocycle 
$c\in Z^1(\Z,L^1(X))$ is such that $$\frac{1}{|n|}\left(c(n)(x)-\int c(n)(x')d\mu(x')\right)\to 0,$$ both a.e.\ and in $L^1$.
By measure-preserving action of $G$ on a probability space $X$, we mean a measurable map $G\times X\to X$, denoted $(g,x)\mapsto g.x$ ($G$ being endowed with the Lebesgue $\sigma$-algebra), such that for every $g,h\in G$, the functions $g.(hx)$ and $(gh)x$ coincide outside a subset of measure zero. This makes $L^p(G)$ a Banach $G$-module for all $1\le p<\infty$.

A generalization of this result is due to Boivin and Derriennic \cite{BD} for $\Z^d$ (and similarly for $\R^d$). To obtain almost sure convergence, stronger integrability conditions are required  when 
 $d>1$ (see \cite[Theorems 1 and 2]{BD}).  Here however, we focus on convergence in $L^1$:
  
\begin{DefnIntro}A CGLC group $G$ satisfies the {\it mean ergodic theorem for cocycles in $L^1$} if for every ergodic measure-preserving action of $G$ on a probability space $X$, and every continuous cocycle $c\in Z^1(G,L^1(X))$, we have $$\lim_{|g|\to \infty}\frac{1}{|g|}\left(c(g)(x)-\int_Xc(g)(x')d\mu(x')\right)=0,$$
where the convergence is in $L^1(X)$, and $|g|$ denotes a word length on $G$ associated to some compact generating subset.
\end{DefnIntro}
In \cite[Theorem 4]{BD}, Boivin and Derriennic prove that $\Z^d$   
 and $\R^d$ satisfy the mean ergodic theorem in $L^1$.
We start by the following observation.
\begin{PropIntro}\label{ergoprop}
A group $G$ with Property $\WAPAP$ satisfies the mean ergodic theorem for 1-cocycles in $L^1$ if and only if $G$ satisfies $\WAPT$. 
\end{PropIntro}

The {\it if} part immediately follows from the well-known fact \cite[Corollary 6.5]{EFH} that the representation of $G$ on $L^1(X)$ is WAP. The ``only if" part is more anecdotical, see \S\ref{aux} for the proof.
\begin{CorIntro}
Groups in the class $\mathfrak{C}$ and their closed cocompact subgroups satisfy the  ergodic theorem for cocycles in $L^1$.
\end{CorIntro}

To our knowledge, this is new even for the group SOL. For nilpotent groups, it can be easily deduced from Proposition \ref{ergoprop} together with the fact that these groups have $\WAPT$, an observation due to \cite{BRS}.

\subsection{Bourgain's theorem on tree embeddings}

We obtain a new proof of the following result of Bourgain.
\begin{CorIntro}[Bourgain, \cite{Bourgain}]\label{corbo}
The $3$-regular tree does not  quasi-isometrically embed into any superreflexive Banach space.
\end{CorIntro}

The idea is to use a CGLC group in the class $\mathfrak{C}$ that is quasi-isometric to the 3-regular tree, and make use of amenability and Property $\WAPT$. In \cite{CTV}, the authors and Valette used a similar argument based on property $\HT$ to show Bourgain's result in the case of a Hilbert space. See Section \ref{aux} for the proof.

\medskip

\noindent {\bf Acknowledgement.} We thank Yehuda Shalom for useful remarks and references.

\section{Preliminaries on Banach modules}

\subsection{First reduced cohomology versus affine actions}

Let $G$ be a locally compact group, and $(V,\pi)$ be a Banach $G$-module.

Observe that, given a continuous function $b:G\to V$, we can define for every $g\in G$ an affine transformation $\alpha_b(g)v=\pi(g)v+b(g)$. Then the condition $b\in Z^1(G,\pi)$ is
a restatement of the condition $\alpha_b(g)\alpha_b(h)=\alpha_b(gh)$ for all $g,h\in G$, meaning that $\alpha$ is an action by affine transformations.
Then the subspace $B^1(G,\pi)$ is the set of $b$ such that $\alpha_b$ has a $G$-fixed point, and its closure $\overline{B^1}(G,\pi)$ is
the set of 1-cocycles $b$ such that the action $\alpha_b$ almost has fixed points, that is, for
every $\eps>0$ and every compact subset $K$ of $G$, there exists a
vector $v\in V$ such that for every $g\in K$,
$$\|\alpha_b(g)v-v\|\leq \eps.$$
If $G$ is compactly generated and if $S$ is a compact generating
subset, then this is equivalent to the existence of a sequence of
almost fixed points, i.e.\ a sequence $v_n$ of vectors satisfying
$$\lim_{n\to\infty}\sup_{s\in S}\|\alpha_b(s)v_n-v_n\|=0.$$

\subsection{Almost periodic actions}

\begin{defn}[Almost periodic actions]\label{dapa}
Let $G$ be a group acting on a metric space $X$. Denote by $X^{G,\mathrm{ap}}$ the set of $x\in X$ whose $G$-orbit has a compact closure in $X$.
Say that $X$ is almost ($G$-)periodic if $X^{G,\mathrm{ap}}=X$.
\end{defn}

Note that this definition does not refer to any topology on $G$. Although we are mainly motivated by Banach $G$-modules, some elementary lemmas can be established with no such restriction.

The following lemma is well-known when $X=V$ is a WAP Banach $G$-module.

\begin{lem}\label{vap}
Let $G$ be a group and $X$ a complete metric space with a uniformly Lipschitz $G$-action (in the sense that $C<\infty$, where $C$ is the supremum over $g$ of the Lipschitz constant $C_g$ of the map $x\mapsto gx$). Then $X^{G,\mathrm{ap}}$ is closed in $X$.
\end{lem}
\begin{proof}
Let $v$ be a point in the closure of $X^{G,\mathrm{ap}}$. Choose $v_j\in X^{G,\mathrm{ap}}$ with $v=\lim_jv_j$.

Let $(g_n)$ be a sequence in $G$. We have to prove that $(g_nv)$ has a convergent subsequence. First, up to extract, we can suppose that $(g_nv_j)$ is convergent for all $j$. Then for all $m,n,j$ we have 
\begin{align*}d(g_nv,g_mv)\le & d(g_nv,g_nv_j)+d(g_nv_j,g_mv_j)+d(g_mv_j,g_mv)\\
\le & 2Cd(v,v_j)+d(g_nv_j,g_mv_j).
\end{align*}
Now fix $\eps>0$ and choose $j$ such that $d(v,v_j)\le\eps/3$. Then there exists $n_0$ such that for all $n,m\ge n_0$, we have $d(g_nv_j,g_mv_j)\le\eps/3$. It then follows from the above inequality that for all $n,m\ge n_0$, we have $d(g_nv,g_mv)\le\eps$.
\end{proof}

\begin{lem}\label{vapcocom}
Let $G$ be a locally compact group and $H$ a closed cocompact subgroup. Let $X$ be a metric space with a uniformly Lipschitz, separately continuous $G$-action. Then $X^{H,\mathrm{ap}}=X^{G,\mathrm{ap}}$.
\end{lem}
\begin{proof}
Choose $x\in X^{H,\mathrm{ap}}$.
Choose a compact subset $K\subset G$ such that $G=KH$. Let $(g_nx)$ be any sequence in the $G$-orbit of $x$. Write $g_n=k_nh_n$ with $k_n\in K$, $h_n\in H$. We can find an infinite subset $I$ of integers such that $k_n\underset{n\to\infty}{{}^{\underrightarrow{n\in I}}}k$ and $h_nv\underset{n\to\infty}{{}^{\underrightarrow{n\in I}}}w$. Then, denoting $C$ the supremum of Lipschitz constants, we have
\begin{align*}d(g_nx,kw)= d(k_nh_nx,kw)\le &d(k_nh_nx,k_nw)+d(k_nw,kw)\\
\le & Cd(h_nx,w)+Cd(w,k_n^{-1}kw)\underset{n\to\infty}{{}^{\underrightarrow{n\in I}}}0.
\end{align*}
Thus the $G$-orbit of $x$ has compact closure, that is, $x\in V^{G,\mathrm{ap}}$, showing the nontrivial inclusion.
\end{proof}

\begin{prop}\label{ap_com}
Let $G$ be a topological group and $(X,\pi)$ a metric space with a separately continuous uniformly Lipschitz $G$-action (denoted $gx=\pi(g)x$). Then $X$ is almost $G$-periodic if and only if the action of $G$ on $X$ factors through a compact group $K$ with a separately continuous action of $K$ on $X$ and a continuous homomorphism with dense image $G\to K$.
\end{prop}
\begin{proof}
The ``if" part is immediate. Conversely, suppose that $X$ is almost periodic. Let $\mathcal{O}$ be the set of $G$-orbit closures in $X$. Each $Y\in\mathcal{O}$ is a compact metric space; let $I_Y$ be its isometry group; we thus have a continuous homomorphism $G\to I_Y$. So we have a canonical continuous homomorphism $q:G\to \prod_{Y\in\mathcal{O}}Y$; let $K$ be the closure of its image; this is a compact group. We claim that the representation has a unique separately continuous extension to $K$. The uniqueness is clear. To obtain the existence, consider any net $(g_i)$ in $g$ such that $(q(g_i))$ converges in $K$, $\pi(g_i)v$ converges in $X$ for every $v\in X$. Indeed, we have $d(\pi(g_i^{-1}g_j)v,v)\to 0$ when $i,j\to\infty$, and hence, since the representation is uniformly bounded, $d(\pi(g_j)v,\pi(g_i)v)$ tends to 0. So $(\pi(g_i)v)$ is Cauchy and thus converges; the limit only depends on $v$ and on the limit $k$ of $(q(g_i))$; we define it as $\tilde{\pi}(k)v$. 

Also, if $c$ is the supremum of all Lipschitz constants, then, as a pointwise limit of $c$-Lipschitz maps, $\tilde{\pi}(k)$ is $c$-Lipschitz.

Now let us show that it defines an action of $K$, namely $\tilde{\pi}(k)\tilde{\pi}(\ell)=\tilde{\pi}(k\ell)$ for all $k,\ell$ in $K$. We first claim that $k\mapsto\tilde{\pi}(k)v$ is continuous for every fixed $v\in X$. Indeed, if $Y$ is the closure of the orbit of $v$, then this map can be identified to the orbital map of the action of the image of $k$ in $I_Y$, which is continuous. The same argument shows that $(k,\ell)\mapsto\tilde{\pi}(k)\tilde{\pi}(\ell)v$ is continuous, and also this implies, by composition, that $(k,\ell)\mapsto\tilde{\pi}(k\ell)v$ is continuous. Since these two maps coincide on the dense subset $q(G)\times q(G)$, they agree. Thus $\tilde{\pi}$ defines an action, and we have also checked along the way that it is separately continuous. 
\end{proof}

\subsection{Almost periodic Banach modules}

We recall the following well-known fact.

\begin{lem}\label{b_ub}
Let $G$ be a group and $(V,\pi)$ a Banach $G$-module. If every $G$-orbit is bounded, then the representation is uniformly bounded. In particular,
\begin{itemize}
\item every WAP representation is uniformly bounded;
\item for $G$ compact, every Banach $G$-module is uniformly bounded.
\end{itemize}

\end{lem}
\begin{proof}
Define $V_n=\{v\in V:\forall g\in G,\|gv\|\le n\}$. Since the action is by bounded (=continuous) operators, $V_n$ is closed for all $n$. Since all $G$-orbits are bounded, $\bigcup_n V_n=V$. By Baire's theorem, there exists $V_n$ with non-empty interior. Since $V_n=-V_{n}$ and $V_k+V_\ell\subset V_{k+\ell}$ for all $k,\ell$, the set $V_{2n}$ contains the centered ball of radius $\eps$ for some $\eps>0$. This implies that $\sup_{g\in G}\|\pi(g)\|\le 2n/\eps$.

(Note that the last consequence was not a trivial consequence of the definition, since the map $g\mapsto\pi(g)$ often fails to be continuous for the norm topology on operators.)
\end{proof}

\begin{lem}\label{vap1}
Let $V$ be a Banach $G$-module. Then $V^{G,\mathrm{ap}}$ is a subspace of $V$.
\end{lem}
\begin{proof}
This is clear since $G(\lambda v+w)\subset \lambda\overline{Gv}+\overline{Gw}$ for all $v,w$.
\end{proof}

\begin{thm}\cite[Theorem 2]{Shi}\label{thm:PW}
Let $G$ be a compact group and $(V,\pi)$ be a Banach $G$-module. Then the sum of finite-dimensional irreducible $G$-submodules of $V$ is dense in $V$.
\end{thm}

Let now $G$ be an arbitrary topological group. Recall that a Banach $G$-module is almost periodic if every $G$-orbit is relatively compact in the norm topology. 

\begin{cor}\label{sumfin}
Let $V$ be a uniformly bounded Banach $G$-module. Then $V^{G,\mathrm{ap}}$ is the closure of the sum of all finite-dimensional submodules of $V$, and is also the closure of the sum all irreducible finite-dimensional submodules of $V$.
\end{cor}
\begin{proof}
Let $V^{G,\mathrm{ap}},V_2,V_3$ be the three subspaces in the corollary. That $V_3\subset V_2$ is clear.

That every finite-dimensional submodule is contained in $V^{G,\mathrm{ap}}$ is clear (even without assuming $V$ uniformly bounded). So the sum of all finite-dimensional submodules is contained in $V^{G,\mathrm{ap}}$, and hence its closure by Lemma \ref{vap}, since $V$ is uniformly bounded. So $V_2\subset V^{G,\mathrm{ap}}$.

For the inclusion $V^{G,\mathrm{ap}}\subset V_3$, we use that the $G$-action on $V^{G,\mathrm{ap}}$ factors through a compact group (Proposition \ref{ap_com}), and then invoke Theorem \ref{thm:PW}.
\end{proof}

\begin{defn}\label{dptscont}
Let $G$ be a locally compact group and let $(V,\pi)$ be a Banach $G$-module. Define $V^*_{[G]}$ as the set of $f\in V^*$ such that the orbital function $\nu_f:g\mapsto g\cdot f$ is continuous on $G$.
\end{defn}

\begin{lem}\label{ptscont}
Let $G$ be a locally compact group and $(V,\pi)$ be a Banach $G$-module. Then $V^*_{[G]}$ is a closed subspace of $V^*$ (and thus is a Banach $G$-module). Moreover, $V^*_{[G]}$ separates the points of $G$.
\end{lem}
\begin{proof}
That $V^*_{[G]}$ is a subspace is clear.

Write $c=\sup_{g\in G}\|\pi(g)\|$ (it is finite by Lemma \ref{b_ub}).  For $f,f'\in V^*$ and $v\in V$, we have \[\nu_f(g)(v)-\nu_{f'}(g)(v)=(g\cdot f)(v)-(g\cdot f')(v)=(f-f')(g^{-1}v);\]
hence
\[\|\nu_f(g)(v)-\nu_{f'}(g)(v)\|\le \|f-f'\|\|g^{-1}v\|\le c\|f-f'\|\|v\|\]
and thus 
\[\|\nu_f(g)-\nu_{f'}(g)\|\le c\|f-f'\|\]
Suppose that $(f_n)$ converges to $f$, with $f_n\in V^*_{[G]}$ and $f\in V^*$. The above inequality shows that $\nu_{f_n}$ converges uniformly, as a function on $G$, to $\nu_f$. By assumption, $\nu_{f_n}$ is continuous, and thus $\nu_f$ is continuous, meaning that $f\in V^*_{[G]}$.

For the separation property, we consider $v\in V\smallsetminus\{0\}$ and have to find an element in $V^*_{[G]}$ not vanishing on $v$. First, we choose $f\in V^*$ such that $f(v)=1$. 

For every $g\in G$, $(g\cdot f)(v)=f(g^{-1}v)$, and hence $q:g\mapsto (g\cdot f)(v)$ is continuous on $G$; we have $q(1)=1$. Let $U$ be a compact neighborhood of $1$ in $G$ on which $q$ takes values $\ge 1/2$. Let $\varphi$ be a non-negative continuous function on $G$, with support in $U$, and integral 1 ($G$ being endowed with a left Haar measure).

For $\xi\in V$, define $u(\xi)=\int_G \varphi(g)f(g^{-1}\xi)dg$. Then $u$ is clearly linear, and $\|u\|\le c\|f\|$, so $u$ is continuous. We have $u(v)=\int_G h_nq\ge 1/2$. 

It remains to show that $u\in V^*_{[G]}$. It is enough to check that $h\mapsto h\cdot u$ is continuous at 1. We have, for $h\in G$ and $\xi\in V$
\begin{align*}u(\xi)-(h\cdot u)(\xi)= 
& \int_G \varphi(g)f(g^{-1}\xi)dg-\int_G \varphi(g)f(g^{-1}h^{-1}\xi)dg\\
=& \int_G \varphi(g)f(g^{-1}\xi)dg-\int_G \varphi(h^{-1}g)f(g^{-1}\xi)dg\\
=& \int_G (\varphi(g)-\varphi(h^{-1}g))f(g^{-1}\xi)dg.
\end{align*}
Define $\eps_h=\sup_{g\in G}|\varphi(g)-\varphi(h^{-1}g)|$. Since $\varphi$ has compact support, it is uniformly continuous. Hence $\eps_h$ tends to 0 when $h\to 1$. We conclude
\[\|u(\xi)-(h\cdot u)(\xi)\|\le \eps_h\int_G|f(g^{-1}\xi)|dg\le\eps_hc\|f\|\|\xi\|,\]
so $\|u-h\cdot u\|\le \eps_hc\|f\|$, which tends to 0 when $h\to 1$.
\end{proof}
 
\begin{prop}\label{prop:complement}
Let $V$ be an almost periodic Banach $G$-module. Then every finite-dimensional submodule is complemented in $V$ as $G$-module.
\end{prop}
\begin{proof}
By Proposition \ref{ap_com}, we can suppose that $G$ is compact.

Let $C$ be a finite-dimensional submodule; let us show, by induction on $d=\dim(C)$, that $C$ is complemented. This is clear if $d=0$; assume now that $C$ is irreducible. Beware that $V^*$ need not be a Banach $G$-module ($G$ does not always act continuously). We consider the subspace $V^*_{[G]}\subset V$ of Definition \ref{dptscont}, which is a Banach $G$-module by Lemma \ref{ptscont}. Let $F\subset V^*_{[G]}$ the sum of all irreducible finite-dimensional submodules. By Theorem \ref{thm:PW}, $F$ is dense in $V^*_{[G]}$, and by Lemma \ref{ptscont}, the latter separates the points of $V$. So there exists an element in $F$ that does not vanish on $C$. In turn, this means that there is an irreducible finite-dimensional submodule $M$ of $F$ that does not vanish on $C$, or equivalently whose orthogonal $W$ does not contain $C$. Note that $W\subset V$ is closed and that $M$ is isomorphic, as $G$-module, to the dual of $V/W$; in particular, $V/W$ is an irreducible $G$-module; in other words, $W$ is maximal among proper $G$-submodules of $V$. It follows that $V=C\oplus W$. 

If $C$ is not irreducible, let $C'$ be a nonzero proper submodule. Then by induction, $C/C'$ is complemented in $V/C'$, which means that $V=C+W$ with $W$ a $G$-submodule and $W\cap C=C'$. By induction, we can write $W=C'\oplus W'$ with $W'$ a submodule. Hence $V=C\oplus W'$.
\end{proof}

\begin{prop}\label{fcodim:compl}
Let $V$ be an almost periodic Banach $G$-module. Then every finite-codimensional submodule is complemented in $V$ as $G$-module.
\end{prop}
\begin{proof}
Let $W$ be a submodule of finite codimension. Lift a basis of the quotient to a family in the complement $(e_1,\dots,e_n)$. Then there exists an open ball $B_i$ around $e_i$ such that for every $(e'_1,\dots,e'_n)\in\prod B_i$, the family $(e'_1,\dots,e'_n)$ projects to a basis of $V/W$. Since by Corollary \ref{sumfin}, the union of finite-dimensional submodules of $V$ is dense, we can choose $e'_i$ to belong to a finite-dimensional submodule $F_i$, and define $F=\sum F_i$.

Then $F$ is a finite-dimensional $G$-submodule and $F+W=V$. Then $G$ preserves a scalar product on $F$, so preserves the orthogonal $F'$ of $F\cap W$ for this scalar product. Thus $V=F'\oplus W$.
\end{proof}

\subsection{Canonical decompositions}\label{sec:decoca}

We use the following known results.

\begin{thm}\label{decoca}
Let $G$ be a locally compact group and $V$ a WAP $G$-module. Then
\begin{enumerate}
\item\label{jlg1} $V^G$ has a canonical complement, consisting of those $v$ such that 0 belongs to the closure of the convex hull of the orbit $Gv$;
\item\label{jlg2} $V^{G,\mathrm{ap}}$ has a canonical complement, consisting of those $v$ such that 0 belongs to the weak closure of the orbit $Gv$.
\end{enumerate}
\end{thm}

Here by complement of a subspace $W_1$ in a Banach space $V$ we mean a closed submodule $W_2$ such that the canonical map $W_1\oplus W_2\to V$ is an isomorphism of Banach spaces. The complement being here defined in a ``canonical" way, it follows that if $G$ preserves these complements.

These statements are Theorems 14 and 12 in \cite{BRS}. The part (\ref{jlg1}) is due to Alaoglu-Birkhoff \cite{AB} in the special case superreflexive $G$-modules (that is, whose underlying Banach space is superreflexive), and \cite[Proposition 2.6]{BFGM} in general. Part (\ref{jlg2}) is a generalized version of a theorem of Jacobs and de Leeuw-Glicksberg, stated in this generality in \cite{BJM}.

\section{Induction}

\subsection{A preliminary lemma}
\begin{lem}\label{lem:CocycleContinous}
Let $G$ be a locally compact group, let $(E,\pi)$ be a continuous Banach $G$-module. Let $b:G\to E$ satisfy the cocycle relation 
$b(gh)=\pi(g)b(h)+b(g)$. If $b$ is measurable and locally integrable, then it is continuous.
Moreover if a sequence $(b_n)$ in $Z^1(G,\pi)$ is locally uniformly bounded, i.e.\ for every compact subset $K\subset G$, we have $\sup_{g\in K,n\in\N}\|b_n(g)\|<\infty$, and if $(b_n)$ converges pointwise to $b$, then the convergence is uniform on compact subsets.
\end{lem}
\begin{proof}
Pick a probability measure $\nu$ on $G$ with compactly supported continuous density $\phi$ and define $\tilde{b}(g)=\int b(h)(\phi(g^{-1}h)-\phi(h))dh$. One checks (see the proof of Lemma 5.2 in \cite{tesLp}) that  $\tilde{b}$ is continuous and 
that $\tilde{b}(g)-b(g)=\pi(g)v-v$, where $v=\int b(h)d\nu(h)$. Hence $b$ is continuous. 

For the second statement, we first note that $v_n=\int b_n(h)d\nu(h)$ converges to $v=\int b(h)d\nu(h)$, by Lebesgue's dominated convergence theorem. So we are left to consider $(\tilde{b}_n)$, which converges pointwise for the same reason. We conclude observing that the $\tilde{b}_n$ are equicontinuous. Indeed,  let $g_1,g_2\in G$ and define $$C= \sup\{\|b_n(g)\|; \; g\in g_1\supp(\phi)\cup  g_2\supp(\phi), n\in \N\}.$$ We have $\|\tilde{b}_n(g_1)-\tilde{b}_n(g_2)\|$
$$=\left\|\int b_n(h)(\phi(g_1^{-1}h)-\phi(g_2^{-1}h))dh\right\|\leq C\sup_{h\in G}|\phi(g_1^{-1}h)-\phi(g_2^{-1}h)|,$$
and we conclude thanks to the fact that $g\mapsto \phi(g^{-1})$ is uniformly continuous.
\end{proof}

\subsection{Measure equivalence coupling}\label{mec}

For the notions introduced in this subsection, we refer to \cite{Shalom}.

\subsubsection{ME coupling and ME cocycles}\label{mecmec}

Given countable discrete groups $\Gamma$ and $\Lambda$, a measure equivalence (ME) coupling between them is a nonzero $\sigma$-finite measure space $(X,\mu)$, which admits commuting free $\mu$-preserving actions of  $\Gamma$ and $\Lambda$  which both have finite-measure fundamental domains, respectively  $X_{\Gamma}$ and $X_{\Lambda}$. Let $\alpha:\Gamma\times X_\Lambda\rightarrow \Lambda$  
 (resp.\ $\beta:\Lambda\times X_\Gamma\rightarrow \Gamma$) be the corresponding cocycle defined by the rule: for all $x\in X_{\Lambda}$, and all 
 $\gamma\in \Gamma$, $\alpha(\gamma, x)\gamma x \in X_\Lambda$ (and symmetrically for $\beta$). If,  for any $\lambda\in \Lambda$ and 
 $\gamma\in \Gamma$, there exists finite subsets $A_{\lambda}\subset \Gamma$ and $B_{\gamma}\subset \Lambda$ such that
 $\lambda X_{\Gamma}\subset A_{\lambda}X_{\Gamma}$ and  $\gamma X_{\Lambda}\subset B_{\gamma}X_{\Lambda}$, then we say the coupling is uniform, and call it a UME coupling, in which case the groups $\Gamma$ and $\Lambda$ are called UME. 
 We now introduce the following reinforcement of UME.
 
 \begin{defn}\label{def:RCE}
  A random cocompact embedding of $\Lambda$ inside $\Gamma$ is a UME coupling satisfying in addition $X_{\Gamma}\subset X_{\Lambda}$.
 \end{defn} 

\subsubsection{Induction of WAP modules}
We assume that $\Lambda$ and $\Gamma$ are ME and we let $\alpha:\Gamma\times X_{\Lambda}\to \Lambda$ be the corresponding cocycle. Now let $(V,\pi)$ be an $\Lambda$-module. 
The induced module is the $\Gamma$-module $(W,\Ind_{\Lambda}^{\Gamma}\pi)$ defined as follows: $W$ is the space $L^1(X_{\Lambda},\mu,V)$ of measurable maps $f:X_{\Lambda}\to V$ such that 
$$\int_{X_\Lambda}|f(x)|d\mu(x)<\infty,$$
and we let $\Gamma$ act on $W$ by 
$$\Ind_{\Lambda}^{\Gamma}\pi(g)f(x)=\pi(\alpha(g,x))(f(g^{-1}\cdot x).$$
\begin{prop}\label{prop:inducedWAP}
If $(V,\pi)$ is WAP, then so is $(W,\Ind_{\Lambda}^{\Gamma}\pi)$.
\end{prop}
\begin{proof}
This follows from the main result of \cite{Talagrand}.
\end{proof}

\subsubsection{Induction of cohomology}

For simplicity, we shall restrict our discussion to cohomology in degree one. We now assume that $\Lambda$  and $\Gamma$ are UME, and we let $\alpha$ and $\beta$ be the cocycles associated to some UME-coupling. Assume $(V,\pi)$ is a Banach $\Lambda$-module.  Following \cite{Shalom}, we define a topological isomorphism  $I:H^1(\Gamma,\pi)\to H^1(\Gamma,\Ind_{\Lambda}^{\Gamma}\pi)$ as follows: for every $b\in Z^1(\Lambda,\pi)$ define
$$Ib(g)(x)=b(\alpha(g,x)),$$
for a.e.\ $x\in X$ and all $g\in \Gamma$.
Note that $Ib$ is continuous by Lemma \ref{lem:CocycleContinous}. 

Observe that the UME assumption ensures that $Ib(g)$ has finite norm for all $g\in G$ and therefore is a well-defined $1$-cocycle. Now assume in addition that the coupling satisfies $X_{\Gamma}\subset X_{\Lambda}$ (i.e.\ that $\Lambda$ randomly cocompact embeds inside $\Gamma$).  

Then one can define an inverse  $T$ of $I$, defined for all $c\in Z^1(\Gamma,\Ind_{\Lambda}^{\Gamma}\pi)$ and $h\in \Gamma$ by 
$$Tc(h)= \int_{X_{\Gamma}} c(\beta(h,y))(y)d\mu(y).$$ 
The fact that $I$ and $T$ induce inverse maps in cohomology follows from the proof of \cite[Theorem 3.2.1]{Shalom}.

\subsection{Induction from a closed cocompact subgroup}\label{s:induction}

Let $G$ be a LCSC group and $H$ a closed cocompact subgroup of finite covolume. 

Let $(V,\pi)$ be a WAP $H$-module. Let $\mu$ be a $G$-invariant probability measure on the quotient $G/H$. Let $E=L^2(G/H,V,\mu)$ be the space of Bochner-measurable functions $f:G/H\to V$ such that $\|f\|\in L^2(G/H,\mu)$. Let $D\subset G$ be a bounded fundamental domain for the right action of $H$ on $G$; let $s:G/H\to D$ be the measurable section. Define the cocycle $\alpha: G\times G/H\to H$ by the condition that 
$$\alpha(g,x)=\gamma \Longleftrightarrow g^{-1} s(x)\gamma\in D,$$
We can now define a $G$-module $(E,\Ind_{H}^G\pi)$ {\it induced} from $(V,\pi)$ by letting an element $g\in G$ act on $f\in E$ by 
$$(gf)(kH)=\pi(\alpha(g,kH))f(g^{-1}kH).$$
The fact that the induced representation is WAP follows from \cite{Talagrand}. 

Note that one can similarly induce affine actions (the same formula holds, replacing $\pi$ by an affine action $\sigma$). The corresponding formula for $1$-cocycles (corresponding to the orbit of $0$) is as follows: given $b\in Z^1(H,\pi)$, one defines the induced cocycle $\tilde{b}\in Z^1(G,\Ind_{H}^G\pi)$ by 
$$\tilde{b}(h)(gH)= b(\alpha(h,gH)).$$
This defines a continuous cocycle by Lemma \ref{lem:CocycleContinous}.

The map $b\to \tilde{b}$ induces a topological isomorphism in 1-cohomology (by Lemma \ref{lem:CocycleContinous}). The inverse is defined as follows: given a cocycle $c\in Z^1(G,\Ind_{H}^G\pi)$, one gets a cocycle  $\bar{c}\in Z^1(H,\pi)$ by averaging  (see for instance \cite[Theorem 3.2.2]{Shalom}):
\begin{equation}\label{eq:average}
\bar{c}(\gamma)=\int_{D}c(x\gamma x^{-1})(x)d\mu(x).
\end{equation}
Observe that that $\bar{\tilde{b}}=b.$

\section{Properties $\WAPT$ and $\WAPAP$}\label{propwa}

\subsection{The definitions}
\begin{defn}Let $G$ be a locally compact group. We say that $G$ has
\begin{itemize}
\item Property $\WAPT$ if $\overline{H^1}(G,V)=0$ for every WAP Banach $G$-module $V$ such that $V^G=0$;
\item Property $\WAPAP$ if $\overline{H^1}(G,V)=0$ for every WAP Banach $G$-module $V$ such that $V^{G,\mathrm{ap}}=0$;
\item Property $\WAPFD$ if for every $G$-module $V$ and $b\in Z^1(G,V)$ that is not an almost coboundary, there exists a closed $G$-submodule of positive finite codimension modulo which $b$ is unbounded.
\item Property $\APFD$: same as $\WAPFD$, but assuming that $V$ is almost periodic.
\end{itemize}
\end{defn}

There is a convenient restatement of the definition of $\WAPFD$, in view of the following lemma:

\begin{lem}
Let $V$ be an almost periodic Banach $G$-module and $b$ a 1-cocycle. The following are equivalent:
\begin{enumerate}
\item\label{ape1} there exists a $G$-module decomposition $V=V_1\oplus V_2$ such that $\dim(V_1)<\infty$ and, under the corresponding decomposition $b=b_1+b_2$, we have $b_1$ unbounded;
\item\label{ape2} there is a closed $G$-submodule $W\subset V$ of positive finite codimension such that the projection of $b$ in $V/W$ is unbounded.
\end{enumerate}
\end{lem}
\begin{proof}
Clearly (\ref{ape1}) implies (\ref{ape2}), and the converse follows from the fact that $W$ is complemented (Proposition \ref{fcodim:compl}).
\end{proof}

\begin{prop}\label{wapfd_conj}
Property $\WAPFD$ is equivalent to the conjunction of Properties $\WAPAP$ and $\APFD$.
\end{prop}
\begin{proof}
It is obvious that $\WAPFD$ implies both other properties. Conversely, assume that $G$ has both latter properties. Let $V$ be a $G$-module and let $b$ be a 1-cocycle that is not an almost coboundary. By Theorem \ref{decoca}(\ref{jlg2}), we have $V=V^{G,\mathrm{ap}}\oplus W$ for some $G$-submodule $W$. Let $b=b_1+b_2$ be the corresponding decomposition of $b$. Since $W^{G,\mathrm{ap}}=0$, Property $\WAPAP$ implies that $b_2$ is an almost coboundary. Hence $b_1$ is not an almost coboundary for the almost periodic $G$-module $V^{G,\mathrm{ap}}$. Then Property $\APFD$ yields the conclusion.
\end{proof}

\begin{prop}
$G$ has Property $\WAPT$ (resp.\ $\WAPAP$) if and only if, for every $G$-module $V$ and $b\in Z^1(G,V)$ that is not an almost coboundary, $V^G\neq 0$ (resp.\ $V$ admits a nonzero finite-dimensional submodule).
\end{prop}
\begin{proof}
The case $\WAPT$ is trivial and only stated to emphasize the analogy.

Suppose that $G$ satisfies the given property (in the second case). Let $V$ be a WAP $G$-module with $V^{G,\mathrm{ap}}=0$. Let $b$ be a 1-cocycle. Since the condition $V^{G,\mathrm{ap}}=0$ implies that $V$ has no nonzero finite-dimensional subrepresentation, the assumption implies that $b$ is an almost coboundary.

Conversely, suppose that $G$ has Property $\WAPAP$. Let $V$ and $b$ be as in the assumptions. By Theorem \ref{decoca}(\ref{jlg2}), write $V=V^{G,\mathrm{ap}}\oplus W$ with $W$ its canonical complement. Decompose $b=b_1+b_2$ accordingly. Then by Property $\WAPAP$, $b_2$ is an almost coboundary. So $b_1$ is not a coboundary. Hence $V^{G,\mathrm{ap}}\neq 0$. Hence, it admits a nonzero finite-dimensional subrepresentation, by Corollary \ref{sumfin}.
\end{proof}

As a consequence, we have the implications
\[\WAPT\Rightarrow\WAPFD\Rightarrow\WAPAP.\]
The left-hand implication is not an equivalence, for instance the infinite dihedral group is a counterexample. 

\begin{que}Are Properties $\WAPFD$ and $\WAPAP$ equivalent?\label{qwaeq}
\end{que}

This is the case for the unitary Hilbert analogue, because any almost periodic unitary Hilbert $G$-module can be written as an $\ell^2$-direct sum of finite-dimensional ones. A positive answer to the question, even with some restrictions on the class of $G$-modules considered, would be interesting (at least if the given class has good stability properties under induction of actions).

In view of Proposition \ref{wapfd_conj}, a positive answer would follow from a positive answer to:

\begin{que}\label{qapfd}
Does Property $\APFD$ hold for an arbitrary locally compact group $G$?
\end{que}

In turn, a positive answer would result from the following less restrictive question:

\begin{que}
Let $G$ be a locally compact group and $V$ an almost periodic Banach $G$-module. Consider $b\in Z^1(G,V)\smallsetminus \overline{B^1(G,V)}$. Does there exist a $G$-submodule of finite codimension $W\subset V$ such that the image of $b$ in $Z^1(G,V/W)$ is unbounded?
\end{que}
For instance, the answer is positive in the case of unitary Hilbert $G$-modules.

See \S\ref{pAPFD} for more on Property $\APFD$.

\subsection{Extension by a compact normal subgroup}

\begin{prop}\label{prop:compactExt}
Let $1\to K\to G\to Q\to 1$ be a short exact sequence of locally compact groups, and assume that $K$ is compact. Then $G$ has property $\WAPT$ (resp.\ $\WAPAP$, resp.\ $\WAPFD$) if and only if $Q$ does.
\end{prop}
\begin{proof}
These properties are obviously stable under taking quotients.

For the converse, consider a WAP $G$-module $(V,\pi)$. Let $b$ be a 1-cocycle. Since $K$ is compact, we can find a cohomologous 1-cocycle $b'$ that vanishes on $K$. Then $b'$ takes values in $V^K$: indeed, for $g\in G$ and $k\in K$, $b(kg)=\pi(k)b(g)+b(k)=\pi(k)b(g)$, so 
\[\pi(k)b(g)=b(kg)=b(gg^{-1}kg)=\pi(g)b(g^{-1}kg)+b(g)=b(g).\] 

If we assume that $Q$ has Property $\WAPT$, and we assume $V^G=0$, then $(V^K)^G=0$ and we deduce that $b$ is an almost coboundary, showing that $G$ has Property $\WAPT$. If we assume that $Q$ has Property $\WAPAP$, and assume $V^{G,\mathrm{ap}}=0$, we deduce that $(V^K)^{Q,\mathrm{ap}}=0$. It follows that $b'$ is an almost coboundary, and hence $b$ as well.

If we assume that $Q$ has Property $\WAPFD$, we first invoke Theorem \ref{decoca}(\ref{jlg1}): we have $V=V^K\oplus W$, where $W$ is a canonically defined complement. Then by Property $\WAPFD$ of $Q$, we have $V^K=W'\oplus W''$, where $W'$ is a finite-dimensional $G$-submodule and $b$ has an unbounded projection on $W'$ modulo $W''$. This shows that $G$ has Property $\WAPAP$.
\end{proof}

\subsection{Invariance of $\WAPT$ under central extension}
The material of this section uses some trick which was exploited in \cite{ANT} in the case of Heisenberg's group. See \cite[Theorem 4.1.3]{Shalom} in the Hilbert setting and \cite[Theorem 2]{BRS} for a more general statement (involving reduced cohomology in any degree).

\begin{prop}\label{prop:ni}
Let $G$ be a locally compact group with a compactly generated, closed central subgroup $Z$ such that $G/Z$ has Property $\WAPT$. Then $G$ has Property $\WAPT$.
\end{prop}
\begin{proof}
We start with the case when $Z$ is discrete cyclic. Let $1\to \Z\to G\to Q\to 1$ be a central extension where $Q$ has property $\WAPT$. Let $V$ be a weakly almost periodic Banach space and let $(V,\pi)$ be a $G$-module with $V^G=0$, and let $b$ be a cocycle. Let $V^{\Z}$ be the subspace of $V$ consisting of fixed $\Z$-vectors. Because $\Z$ is central,  $V$ decomposes as a $G$-invariant direct sum $V=V_1\oplus V_2$, where $V_1=V^{\Z}$ and $V_2$ is its canonical complement (Proposition \ref{decoca}(\ref{jlg1})). 

Let us decompose $\pi=\pi_1\oplus \pi_2$ and $b=b_1\oplus b_2$ accordingly, with $b_i\in Z^1(G,\pi_i)$.
Since $\overline{H}^1(G,\pi)=\oplus_i \overline{H}^1(G,\pi_i)$, it is enough to show that both terms in the direct sum vanish. Let $z$ be a generator of $\Z$. 
Let us first show that $\overline{H}^1(G,\pi_2)=0$, showing that under this assumption the sequence $x_n=\frac{1}{n}\sum_{i=0}^{n-1}b(z^i)$ is almost fixed by the affine action $\sigma$ of $G$ associated to $b$. The cocycle relation together with the fact that $\Z$ is central imply that

$$\sigma(g)x_n =  \frac{1}{n}\sum_{i=0}^{n-1}b(gz^i)
                    =  \frac{1}{n}\sum_{i=0}^{n-1}b(z^ig)
                   = \frac{1}{n}\sum_{i=0}^{n-1}\pi(z)^ib(g)
$$
which goes to zero by the ergodic theorem (which states in general that this converges to a $G$-invariant vector and is an immediate verification).
Let $\phi:\overline{H}^1(G,\pi_1)\to \overline{H}^1(\Z,\pi_1)$ be the map in cohomology obtained by restricting cocycles to the central subgroup $\Z$. Using that $\Z$ is central and that the restriction of $\pi_1$ to $\Z$ is trivial, we deduce that $b_1(z)$ is a $\pi(G)$-invariant vector, which therefore equals $0$. It follows that $b_1$ induces a cocycle $\tilde{b}_1$ for the representation $\tilde{\pi}_1$ of $Q$. It is easy to see that $b_1$ is an almost coboundary if and only if  $\tilde{b}_1$ is an almost coboundary. So we conclude thanks to the fact that $Q$ has $\WAPT$. 

Let us now prove the general case.
As a CGLC abelian group, $Z$ has a cocompact discrete subgroup $\Lambda$ isomorphic to $\Z^d$ for some $d$. Then $G/Z$ is quotient of $G/\Lambda$ with compact kernel, and hence by Proposition \ref{prop:compactExt}, $G/\Lambda$ has Property $\WAPT$.
Then by an iterated application of the case with discrete cyclic kernel, $G$ has Property $\WAPT$ as well.
\end{proof}

\begin{cor}
Among compactly generated locally compact groups, the class of compactly presented groups with Property $\WAPT$ is closed under taking central extensions.
\end{cor}
\begin{proof}
Let $G$ be compactly generated, with a central subgroup $Z$ such that $G/Z$ is also compactly presented and has Property $\WAPT$. Since $G/Z$ is compactly presented and $G$ is compactly generated, $Z$ is compactly generated. Hence Proposition \ref{prop:ni} applies (and $G$ is compactly presented).
\end{proof}

Since CGLC nilpotent groups are compactly presented, we deduce

\begin{cor}\label{nilwapt}[\cite{BRS}, Theorem 8]
CGLC nilpotent groups have Property $\WAPT$.
\end{cor}

\subsection{Cocompact subgroups}

\begin{lem}\label{almostcococo}
Let $G$ be a locally compact group, $H$ a closed normal cocompact subgroup. Let $(V,\pi)$ be a Banach $G$-module. If $b\in Z^1(G,V)$ is an almost coboundary in restriction to $H$, then it is an almost coboundary.
\end{lem}
\begin{proof}
Use a bounded measurable section $P\subset G$, so that $P\times H\to G$ is a measurable bijective map (with measurable inverse). Denote by $y\mapsto\hat{y}$ the section $G/H\to P$. Let $(v_n)$ be a sequence of $H$-almost fixed vectors. Let $S$ be a compact generating subset of $H$; enlarging $S$ if necessary, we can suppose that $PS$ contains a compact generating subset of $G$.

Define $\xi_n=\int_{x\in P}\alpha(\hat{x})v_ndx$. Then for $y\in G/H$ and $s\in S$, we have 
\begin{align*}\xi_n-\alpha(\hat{y}s)\xi_n=&
\int_{x\in P}\alpha(x)v_ndx-\alpha(\hat{y}s)\int_{x\in P}\alpha(x)v_ndx\\
=& \int_{x\in G/H}\alpha(\hat{x})v_n-\int_{x\in G/H}\alpha(\hat{y}s\hat{x})v_ndx\\
=& \int_{x\in G/H}\alpha(\hat{x})v_n-\int_{x\in G/H}\alpha(\hat{y}s\widehat{y^{-1}x})v_ndx\\
=& \int_{x\in G/H}(\alpha(\hat{x})v_n-\alpha(\hat{x}\nu(x,y,s))v_n)dx\\
=&\int_{x\in G/H}\pi(\hat{x})(v_n-\alpha(\nu(x,y,s))v_n)dx,
\end{align*}
with $\nu(x,y,s)=\hat{x}^{-1}\hat{y}s\widehat{y^{-1}x}$. Since $\nu(x,y,s)$ belongs to some fixed ball of $H$ (independently of $x,y,s$), we have $\|v_n-\alpha(\nu(x,y,s))v_n\|\le\eps_n$ for some sequence $(\eps_n)$ depending only of this ball, tending to zero.
 Thus
 $\|\xi_n-\alpha(\hat{y}s)\xi_n\|\le c\eps_n$, where $c=\sup_{g\in G}\|\pi(g),\|$
and hence $(\xi_n)$ is a sequence of almost fixed points for $G$.
\end{proof}

\begin{prop}\label{prop:subgroup}
Let $G$ be a locally compact group and $H$ a closed cocompact subgroup.
\begin{enumerate}
\item\label{ps1} if $H$ has Property $\WAPAP$, so does $G$;
\item\label{ps2} if $H$ has Property $\WAPT$ and is normal in $G$, then $G$ has Property $\WAPFD$.
\end{enumerate}
\end{prop}
\begin{proof}
Suppose that $H$ has Property $\WAPAP$. Let $V$ be a $G$-module with $V^{G,\mathrm{ap}}=0$ and $b\in Z^1(G,\pi)$. Since $H$ is cocompact in $G$, we have $V^{H,\mathrm{ap}}\subset V^{G,\mathrm{ap}}$ and hence $V^{H,\mathrm{ap}}=0$. By Property $\WAPAP$ of $H$, $b$ is an almost coboundary in restriction to $H$, and hence on $G$ by Lemma \ref{almostcococo}. Hence $G$ has Property $\WAPAP$.

Now assume that $H$ is normal and has Property $\WAPT$.
Let $V$ be a WAP $G$-module and $b$ a 1-cocycle that is not an almost coboundary.
Since $H$ is normal, Theorem \ref{decoca}(\ref{jlg1}) implies that $V$ decomposes as a $G$-invariant direct sum $V=V^H\oplus V_2$. Decompose $b=b_1+b_2$ accordingly. Since $H$ has Property $\WAPT$, $b_2$ is an almost coboundary in restriction to $H$, and hence, by Lemma \ref{almostcococo}, $b_2$ is an almost coboundary on $G$. Hence $b_1$ is not an almost coboundary (on $G$). But $b_1$ is a group homomorphism in restriction to $H$, and since $H$ is CGLC, $b_1(H)$ generates a finite-dimensional subspace $F$ of $V$. Since $H$ is normal, this subspace is $\pi(G)$-invariant. By Proposition \ref{prop:complement}, $F$ as a complement $W$ in $V^H$ as a $G$-module, and under the decomposition $V=F\oplus (W\oplus V_2)$ we have $b=b_1+(0+b_2)$, where $b_1$ is unbounded. This shows Property $\WAPFD$.
\end{proof}

\begin{thm}\label{thm:subgroup}
Properties $\WAPT$ and $\WAPFD$ are inherited by closed cocompact subgroups $H\subset G$ of finite covolume. 
\end{thm}

\begin{proof}
We start with Property $\WAPT$. Let $(V,\pi)$ be a WAP $H$-module and $c$ a 1-cocycle. We use the notation of \S\ref{s:induction}; in particular $(E,\Ind_{H}^G\pi)$ is the induced $G$-module.
Assuming that $c$ is nonzero in the reduced cohomology, we get that $\tilde{c}$ is also nonzero in reduced cohomology. Decompose the cocycle $\tilde{c}=\tilde{c}_1+\tilde{c}_2$ according to the decomposition $E=E_1\oplus E_2$ (see Theorem \ref{decoca}(\ref{jlg1})). Since $G$ has $\WAPT$, and $\tilde{c}$ is nonzero in reduced cohomology, we obtain that $\tilde{c}_1$ is a nonzero group homomorphism. 
Therefore, integrating $\tilde{c}_1$ over $\mu$ as in (\ref{eq:average}) gives back a non-zero group homomorphism $H\to V$ which, being a 1-cocycle, is valued in $V^H$. Hence $V^H\neq 0$, proving Property $\WAPT$.

Now suppose $G$ has $\WAPFD$. We argue in the same way, but instead with $E_1=E^{G,\mathrm{ap}}$ (using Theorem \ref{decoca}(\ref{jlg2}) instead). We obtain a decomposition $E=F\oplus W$ of $G$-module, with $F$ finite-dimensional, such that the corresponding decomposition $\tilde{c}=\tilde{c}_1+\tilde{c}_2$ has $\tilde{c}_1$ unbounded (hence not an almost coboundary). 
Write $c_i=\overline{\tilde{c}_i}\in Z^1(H,\pi)$. Then $c_1$ is also not an almost coboundary, and in addition, has its range contained in a finite-dimensional subspace. Clearly the subspace $V_1$ spanned by the range of $c_1$, being the affine hull of the orbit of 0 in the affine action defined by $c_1$, is $\pi$-invariant. By Proposition \ref{prop:complement}, we can find an $H$-module complement $V=V_1\oplus V_2$, and the projection of $c_1$ is just $c_1$. On the other hand, $\tilde{c_2}$ being an almost boundary, so is $c_2$, as well as its projections. Since $c=\bar{\tilde{c}}$, we have $c=c_1+c_2$, and the projection of $c$ to $V_1$ differs from $c_1$ by a bounded function, and hence is unbounded. This proves that $H$ has Property $\WAPFD$.
\end{proof}

\begin{rem}
We could not adapt this proof to Property $\WAPAP$. 
\end{rem}

\begin{thm}\label{polgr}
Every compactly generated, locally compact group $G$ with polynomial growth has Property $\WAPFD$.
\end{thm}
\begin{proof}
By Losert's theorem \cite{Lo} (due to Gromov in the discrete case), for such $G$, there exists a copci (proper continuous with cocompact image) homomorphism to a locally compact group of the form $N\rtimes K$ with $N$ a simply connected nilpotent Lie group and $K$ a compact Lie group. Let $W$ be the kernel of such a homomorphism.

Then $N$ has Property $\WAPT$ by Corollary \ref{nilwapt}, and hence $N\rtimes K$ has Property $\WAPFD$ by Proposition \ref{prop:subgroup}(\ref{ps2}), and hence $G/W$ has Property $\WAPFD$ by Theorem \ref{thm:subgroup}, and in turn $G$ has Property $\WAPFD$ by Proposition \ref{prop:compactExt}.
\end{proof}

\subsection{Stability under RCE}

\begin{thm}\label{thm:stabRCE}
If a countable group $\Lambda$ has Property $\WAPFD$ and  $\Lambda$ randomly cocompactly embeds inside another countable group $\Gamma$, then $\Lambda$ has  $\WAPFD$ as well. \end{thm}
\begin{proof}
The proof is almost identical to that of Theorem \ref{thm:subgroup} but we reproduce it for the sake of completeness. Assume that $\Lambda$ randomly cocompactly embeds inside $\Gamma$ and that $\Gamma$ has Property $\WAPFD$. Consider a WAP $\Lambda$-module $(V,\pi)$ and a $1$-cocycle $c\in \Z^1(\Lambda,\pi)$. 
Assuming that $c$ is nonzero in the reduced cohomology, we get that $\tilde{c}:=Ic$ is also nonzero in reduced cohomology. By Proposition \ref{prop:inducedWAP}, the induced $\Gamma$-module $(E,Ind_{\Lambda}^{\Gamma}\pi)$ is WAP. Hence there is a decomposition $E=F\oplus W$ of $\Gamma$-module (see Theorem \ref{decoca}(\ref{jlg2})), with $F$ finite-dimensional, such that the corresponding decomposition $\tilde{c}=\tilde{c}_1+\tilde{c}_2$ has $\tilde{c}_1$ unbounded (hence not an almost coboundary). 
Write $c_i=T\tilde{c}_i\in Z^1(H,\pi)$. Then $c_1$ is also not an almost coboundary, and in addition, has its range contained in a finite-dimensional subspace. Clearly the subspace $V_1$ spanned by the range of $c_1$, being the affine hull of the orbit of 0 in the affine action defined by $c_1$, is $\pi$-invariant. By Theorem \ref{prop:complement}, we can find an $\Lambda$-module complement $V=V_1\oplus V_2$, and the projection of $c_1$ is just $c_1$. On the other hand, $\tilde{c_2}$ being an almost boundary, so is $c_2$, as well as its projections. Since $c=TIc$, we have $c=c_1+c_2$, and the projection of $c$ to $V_1$ differs from $c_1$ by a bounded function, and hence is unbounded. This proves that $H$ has Property $\WAPFD$.
\end{proof}

\subsection{Property $\APFD$}\label{pAPFD}

The following part is notably motivated by Question \ref{qapfd}.

Let $G$ be a locally compact group. Let $\mathcal{K}(G)$ be the intersection of all kernels of continuous homomorphisms into compact groups. Then $\mathcal{K}(G)$ is itself such a kernel (using a product). Note that by the Peter-Weyl theorem, it is also the intersection of all kernels of continuous homomorphisms into the finite-dimensional orthogonal groups $\mathrm{O}(n)$.

Next, define $\KT(G)$ to be the intersection of all kernels of continuous homomorphisms into the finite-dimensional isometry groups $\R^n\rtimes\mathrm{O}(n)$. Clearly, $H=G/\KT(G)$ is the largest quotient of $G$ such that $\KT(H)=1$.

Recall that $g\in G$ is {\it distorted} if there exists a compact subset $S$ of $G$ containing $g$ such that $\lim |g^n|_S/n=0$, where $|\cdot|_S$ is the word length with respect to $S$ (in particular, this includes elements of finite order and more generally elliptic elements, for which $(|g^n|_S)$ is bounded for suitable $S$). 

\begin{prop}
$\mathcal{K}(G)/\KT(G)$ is abelian, and contains no nontrivial element that is distorted in $G/\KT(G)$.
\end{prop}
\begin{proof}
We can suppose that $\KT(G)$ is trivial. So we have to prove that $\KK(G)$ is abelian and has no nontrivial distorted element.

If $u,v\in\KK(G)$ do not commute, we can find $n$ and a continuous homomorphism $G\to\R^n\rtimes\mathrm{O}(n)$ such that $[u,v]$ is not in the kernel. Since $u,v\in\KK(G)$, both are mapped to translations, and we have a contradiction. Also, if $u\in\KK(G)\smallsetminus\{1\}$, we can find a continuous homomorphism as above such that $u$ is not in the kernel, and hence $u$ maps to a translation. Then we have a contradiction since the translation is undistorted in the group of Euclidean isometries.
\end{proof}

This yields a method to ``approach" $G/\KT(G)$ from $G$: first mod out by the closure of the derived subgroup $\overline{[\KK(G),\KK(G)]}$. Then mod out by the closure of the subgroup of the abelian subgroup $\KK\Big(G/\overline{[\KK(G),\KK(G)]}\Big)$ consisting of those elements that are distorted in $G/\overline{[\KK(G),\KK(G)]}$. 

\begin{ex}Let $G$ be a real triangulable Lie group. Then $\KT(G)$ is equal to the derived subgroup (because the derived subgroup is equal to $\KK(G)$ and its elements are at least quadratically distorted).

If $G=\mathbb{G}(\mathbf{Q}_p)$ for some linear algebraic $\Q_p$-group $\mathbb{G}$, let $\mathbb{H}=\mathbb{G}/\mathbb{N}$ be the largest quotient with no simple factor of positive $\Q_p$-rank, with abelian unipotent radical, and whose maximal $\Q_p$-split torus centralizes the unipotent radical. (By Borel-Tits, $G$ is compactly generated if and only if $:\mathbb{H}$ is reductive.) Then $\mathbb{N}(G)$ is contained in $\KT(G)$. (If $G$ is compactly generated, then $\mathbb{H}(\Q_p)$ is compact-by-abelian.)
\end{ex}

\begin{prop}\label{apfd_kt}
$G$ has Property $\APFD$ if and only if $G/\KT(G)$ has Property $\APFD$. If $N$ is any closed normal subgroup contained in $\KT(G)$, this is also equivalent to: $G/N$ has Property $\APFD$. 
\end{prop}
\begin{proof}
It is trivial that Property $\APFD$ passes to quotients, hence it passes from $G$ to $G/N$ and from $G/N$ to $G/\KT(G)$. Now suppose that $G/\KT(G)$ has Property $\APFD$. Let $V$ be an almost periodic $G$-module and $b$ a 1-cocycle that is not an almost coboundary. By Proposition \ref{ap_com}, the $G$-representation factors through a compact group; in particular, it is trivial on $\KK(G)$. So on $\KK(G)$, $b$ is given by a continuous group homomorphism $\KK(G)\to V$. We claim that $b$ vanishes on $\KT(G)$. Assume the contrary by contradiction: pick $g\in\KT(G)$ with $b(g)\neq 0$. 

Then by Lemma \ref{ptscont}, there exists $f\in V^*_{[G]}$ such that $f(b(g))\ge 2$. Since $V^*_{[G]}$ is almost periodic, the union of finite-dimensional submodules is dense (Corollary \ref{sumfin}), and hence there is $f'\in V^*_{[G]}$ inside a finite-dimensional submodule $M\subset V^*_{[G]}$ such that $f'(b(g))\ge 1$. Let $W$ be the orthogonal (for duality) of $M$, it has finite codimension and $b(g)\notin W$. This means that the projection of $b$ in $V/W$ is nonzero. Hence $g$ is not in the kernel of the affine action on $V/W$. Hence $g\notin\KT(G)$, a contradiction.
\end{proof}

Say that $G$ has Property $\APT$ if every almost periodic Banach $G$-module $V$ with $V^G=1$, we have $\overline{H^1}(G,V)=0$.

The same proof also shows:

\begin{prop}\label{apt}
Let $G$ be a locally compact group. Then $G$ has Property $\APT$ if and only if $G/\KT(G)$ has Property $\APT$. In particular, if $G$ has Property $\WAPAP$, it has Property $\WAPFD$ if and only if $G/\KT(G)$ has Property $\APT$.
\end{prop}

Let us now provide information about $\KT(G)$ in more specific examples.

\begin{lem}\label{gkt_poly}
Let $G$ be a compactly generated locally compact group in the class $\mathfrak{C}''$. Then $G/\KT(G)$ has polynomial growth.
\end{lem}
\begin{proof}We first see that $\mathrm{Contr}(G)\subset\KT(G)$. This amounts to showing that in $H=\R^n\rtimes\mathrm{O}(n)$ we have $\KT(H)=\{1\}$: indeed first in $L=\mathrm{O}(n)$ we have $\KT(L)=\{1\}$ because $L$ is compact and hence has closed conjugacy classes; it follows that $\KT(H)\subset\R^n$, and clearly the conjugacy classes of $H$ contained in $\R^n$ are compact.

By Lemma \ref{divcon}, we deduce that $U_{\mathrm{\mathrm{div}}}\subset\KT(G)$. Since the class $\mathfrak{C}''$ is stable under taking quotients, we are reduced to proving that if $G$ belongs to the class $\mathfrak{C}''$ satisfies $U_{\mathrm{\mathrm{div}}}=\{1\}$, then $G$ has polynomial growth. Indeed in this case, $U$ is compact and since $G=UN$ and $N$ has polynomial growth, the conclusion follows.
\end{proof}

The previous lemma is a way to show that $G/\KT(G)$ is small in many relevant examples. In contrast, the following proposition shows that it is often large in the setting of discrete groups.

\begin{prop}
Let $\Gamma$ be a discrete and finitely generated group. Then $\KK(\Gamma)=\KT(\Gamma)$ is the intersection of all finite index subgroups of $G$. 
\end{prop}
\begin{proof}
Denote by $R(\Gamma)$ the intersection of all finite index subgroups of $\Gamma$. Clearly, $\KT(\Gamma)\subset\KK(\Gamma)\subset R(\Gamma)$. The remaining inclusion $R(\Gamma)\subset\KT(\Gamma)$ follows from Malcev's theorem that finitely generated linear groups are residually finite.
\end{proof}

\section{A dynamical criterion for property $\WAPAP$}
This section contains the central ideas of this paper. It culminates with Theorem \ref{prop:main}, which provides dynamical criteria for Properties $\WAPT$ and $\WAPAP$. The latter is designed to apply to groups from the class $\mathfrak{C}$. The following lemmas are the key ingredients. It starts with Lemma \ref{mautner}, an analogue of Mautner's phenomenon.

\subsection{Mautner's phenomenon}

\begin{defn}\label{contr}
Let $G$ be a locally compact group, and let $N$ be a subgroup of $G$. Denote by $\Contr(N)$ the set of elements $g\in G$ such that there exists a sequence $a_n\in N$ such that $a_n^{-1}ga_n\to 1$. Such an element $g$ is called $N$-contracted.
\end{defn}

Note that $\Contr(N)$ is stable under inversion; it is not always a subgroup. 

\begin{lem}\label{mautner}{\bf (Mautner's phenomenon)}
Let $G$ be a locally compact group, $N$ a subgroup, and $L$ the subgroup generated by $\mathrm{Contr}(N)$. Consider a separately continuous, uniformly Lipschitz action of $G$ on a metric space $X$. Then the subspace $X^{N,\mathrm{ap}}$ of almost periodic points (Definition \ref{dapa}) is contained in the subspace $X^L$ of $L$-fixed points.
\end{lem}
\begin{proof}Let $C$ be the supremum of all Lipschitz constants of the $G$-action. Fix $u\in\mathrm{Contr}(N)$, a sequence $(a_n)$ with $a_n\in N$ and $\eps_n=a_n^{-1}ua_n\to 1$. 

Consider $x\in X^{N,\mathrm{ap}}$. Write $ux=a_n\eps_na_n^{-1}x$. Let $J$ be an infinite subset of integers such that $(a_n^{-1}x)_{j\in J}$ converges, say to $x'$. 
Then
\begin{align*}
d(a_n\eps_na_n^{-1}x,x)\le & d(a_n\eps_na_n^{-1}x,a_n\eps_nx')+d(a_n\eps_nx',a_nx')+d(a_nx',x)\\
\le &  Cd(a_n^{-1}x,x')+Cd(\eps_nx',x')+Cd(x',a_n^{-1}x)\underset{n\to\infty}{{}^{\underrightarrow{n\in J}}} 0.
\end{align*}
This shows that $ux=a_n^{-1}\eps_na_nx=x$. 
\end{proof}

Lemma \ref{mautner} generalizes Mautner's phenomenon, as well as \cite[Lemma 5.2.6]{Shalom}; both were written in a more specific context for $G$; in addition in Mautner's result the metric space is a Hilbert space with a unitary representation; in Shalom's result $X$ is an arbitrary metric space with an isometric action; in both cases the result takes the form $X^N\subset X^L$. We will use the following consequence of Lemma \ref{mautner}.

\begin{lem}\label{prop:C}
Let $G$ a locally compact group, and $N$ a subgroup of $G$. Define $M$ as the normal subgroup generated by $\mathrm{Contr}(N)$, and $H=\overline{MN}$.
Consider a separately continuous, uniformly Lipschitz action of $G$ on a metric space $X$. Then

\begin{enumerate}
\item\label{ifhg} if $H=G$, then $X^N=X^G$;
\item\label{ifhco} if $H$ is cocompact in $G$, then $X^{N,\mathrm{ap}}=X^{G,\mathrm{ap}}$. \end{enumerate}
\end{lem}
\begin{proof}
By definition $N$ acts trivially on $X^N$, and by Lemma \ref{mautner}, $M$ acts trivially on $X^N$. Hence, in the context of (\ref{ifhg}), $MN$ is dense and hence $G$ acts trivially on $X^N$.

Assume now as in (\ref{ifhco}), and write $H=\overline{MN}$, which is cocompact. Then by Lemma \ref{mautner}, $M$ acts trivially on $X^{N,\mathrm{ap}}$. In particular $H$ preserves $X^{N,\mathrm{ap}}$, and the $H$-action on $X^{N,\mathrm{ap}}$ factors through $H/M$. Since $N$ has a dense image in $H/M$, it follows that the closure of $H$-orbits in $X^{N,\mathrm{ap}}$ coincide with closure of $N$-orbits, which are compact by assumption. This shows that $X^{N,\mathrm{ap}}=X^{H,\mathrm{ap}}$.

Finally, we have $X^{H,\mathrm{ap}}=X^{G,\mathrm{ap}}$ by Lemma \ref{vapcocom}.
\end{proof}

\subsection{Controlled F\o lner sequences and sublinearity of cocycles}

We now recall some material from \cite{CTV}, also used in \cite{tesLp}.
\begin{defn}\cite{CTV}\label{ctvcontro}
Let $G$ be a locally compact group generated by some compact subset $S$ and equipped with some left Haar measure $\mu$. A sequence of compact subsets $F_n\subset G$ of positive measure is called a controlled F\o lner sequence if either $G$ is compact, or $\Diam(F_n)\to \infty$, and there exists a constant $C\geq 1$ such that for all $n\in \N$ and all $s\in S,$
$$\mu(sF_n\vartriangle F_n)\leq C\frac{\mu(F_n)}{\Diam(F_n)}.$$
\end{defn}

\begin{rem}
Let $G$ be a compactly generated group with a compact generating subset $S$. For $n$, let $f(n)$ be the smallest $m$ such that the $m$-ball contains a compact subset $F$ of positive measure such that $\mu(sF_n\vartriangle F_n)\le \mu(F_n)/n$. Then $G$ is amenable if and only if $f(n)<\infty$ for all $n$, and admits a controlled F\o lner sequence if and only if $\liminf f(n)/n<\infty$. Actually in all examples we are aware of groups with controlled F\o lner sequences, we indeed have $f(n)=O(n)$; notably strong controlled F\o lner sequences of Definition \ref{strongf} satisfy this property.
\end{rem}

The following is \cite[Corollary 3.7]{CTV} in the case of unitary Hilbert modules.

\begin{prop}\label{prop:sublin}{\bf (Sublinearity versus triviality of cocycles)}
Let $G$ be a CGLC group and let $(V,\pi)$ be a uniformly bounded Banach $G$-module. Let $b\in Z^1(G,\pi)$ be a 1-cocycle.

If $b$ is an almost coboundary, then it is sublinear, i.e.\ $\|b(g)\|=o(|g|)$ when $|g|$ goes to infinity. The converse holds if $G$ admits a controlled F\o lner sequence: if $b$ is sublinear, then it is an almost coboundary.

\end{prop}
\begin{proof}Both implications are adaptations of the original proof for unitary Hilbert $G$-modules, up to some technical points, which we emphasize below.

Denote $C=\sup_{g\in G}\|\pi\|$, and by $S$ a compact generating subset of $G$; let $|\cdot|$ be the word length in $G$ with respect to $S$. If $b\in Z^1(G,\pi)$ and $T\subset G$, denote $\|b\|_T=\sup_{s\in S}\|b(s)\|_T$. Then we have, for all $g\in G$, the inequality $\|b(g)\|\le C|g|\|b\|_S$. 

Suppose that $b$ is an almost coboundary. For $\eps>0$, there exists a bounded cocycle $b'$ such that $\|b-b'\|\le\eps/C$ on $S$. Say, $\|b'\|_G\le c_{\eps}$. Then, using the previous inequality for the cocycle $b-b'$, we have for all $g\in G$, $\|b(g)\|\le\|b'(g)\|+\|(b-b')(g)\|\le c+|g|\eps$. Thus, for $|g|\ge c_\eps\eps^{-1}$, we have $\frac{\|b(g)\|}{|g|}\le 2\eps$. 

Now assume that $G$ admits a controlled F\o lner sequence and let us prove the converse. Suppose that $b$ is sublinear.
Let $(F_n)$ be a controlled F\o lner sequence in $G$. We need to define a
sequence $(v_n)\in V^{\N}$ by
$$v_n=\frac{1}{\mu(F_n)}\int_{F_n}b(g)dg.$$

Here is the technical issue: since $V$ is not assumed to be reflexive, we have to be more careful to claim that this integral is well-defined. Since $F_n$ is compact, it follows that on $F_n$, we can write the function $b$ (or any continuous function to a normed space) as a uniform limit of simple functions. This implies that $g\mapsto b(g)$ is Bochner-integrable.

We claim that $(v_n)$ defines a sequence of almost fixed points for
the affine action $\sigma$ defined by $\sigma(g)v=\pi(g)v+b(g)$ (which is equivalent to saying that $b$ is an almost coboundary).
Indeed, we have (noting that $\sigma(s)b(g)=b(sg)$)
\begin{eqnarray*}
\|\sigma(s)v_n-v_n\|& = &\left\|\frac{1}{\mu(F_n)}\int_{F_n}\sigma(s)b(g)dg-\frac{1}{\mu(F_n)}\int_{F_n}b(g)dg\right\|\\
                & =& \left\|\frac{1}{\mu(F_n)}\int_{F_n}b(sg)dg -\frac{1}{\mu(F_n)}\int_{F_n}b(g)dg\right\|\\
                & = &\left\|\frac{1}{\mu(F_n)}\int_{s^{-1}F_n}b(g)dg-\frac{1}{\mu(F_n)}\int_{F_n}b(g)dg\right\|\\
                & \leq & \frac{1}{\mu(F_n)}\int_{s^{-1}F_n\vartriangle
                F_n}\|b(g)\|dg.
\end{eqnarray*}
Since $F_n\subset S^n$, we obtain that
$$\|\sigma(s)v_n-v_n\| \leq  \frac{C}{n}\sup_{|g|_S\leq
n+1}\|b(g)\|,$$ which converges to $0$. 
\end{proof}

\subsection{Combing}
An important feature of the groups studied in this article is the following ``combing" property.

\begin{defn}\label{def:combing}
Let $G$ be a locally compact group generated by some compact subset $S$, and let $H\le G$ be a closed subgroup. 
We say that $G$ is $H$-combable if there exists an integer $k\in \N$  and a constant $C\geq 1$ such that every element $g\in G$ can be written as a word $w=w_1\ldots w_k$ in the alphabet $S\cup H$ 
with \begin{equation}\label{eq:combing}\sum_{i=1}^k|w_i|_S\leq C|g|_S.\end{equation}
\end{defn}

\begin{rem}
It is easy to check that being $H$-combable does not depend on the choice of $S$. Moreover, assuming for convenience that $S$ is symmetric with 1, it is equivalent to the existence of constants $k,\ell$ such that $S^n\subset ((S^{\ell n}\cap H)S)^k$ for all $n$. In most examples, we actually have a stronger property: there exist constants $k,\ell$ such that $S^n\subset ((S^{\ell }\cap H)^n S)^k$ for all $n$.
\end{rem}

If $G$ is $H$-combable with $k$ as above, then $G\subset (HS)^k$. In particular, when $H$ is compact, then $G$ is $H$-combable if and only if $G$ is compact. However, there are many interesting cases where $G$ is $H$-combable with $H$ being nilpotent and $G$ having exponential growth.

The following lemma is immediate, but we emphasize it to show how this property can be used.

\begin{lem}\label{sublinear_comb}
Let $G$ be a CGLC group and $H$ a closed subgroup such that $G$ is $H$-combable. Let $\ell$ be a length function on $G$ that is sublinear on $H$ (with respect to the restriction to $H$ of the word length of $G$). Then it is sublinear on $G$.

In particular, if $H$ is compactly generated and $\ell$ is sublinear on $H$ (with respect to its intrinsic word length), then $\ell$ is sublinear on $G$.\qed
\end{lem}

\subsection{Injectivity of the restriction map in 1-cohomology}

\begin{lem}\label{cohinj}{\bf (Restriction in 1-cohomology)}
Let $G$ be a locally compact group and $N$ a closed subgroup. Suppose that $G$ has a controlled F\o lner sequence and is $N$-combable. Let $V$ be a WAP $G$-module. Then the restriction map $\overline{H^1}(G,V)\to\overline{H^1}(N,V)$ is injective.
\end{lem}
\begin{proof}
We can change the norm to an equivalent $G$-invariant norm.
Consider a cocycle $b\in Z^1(G,\pi)$. Suppose that in restriction to $N$, $b$ is an almost coboundary. Then by Proposition \ref{prop:sublin}, $b$ is sublinear in restriction to $N$, i.e.\ $\|b(a)\|=o(|a|_S)$. Note that $\|b\|$ is sub-additive, because the norm is $G$-invariant. Since $G$ is $N$-combable, one thus deduces from Lemma \ref{sublinear_comb} that
$b$ is sublinear on all of $G$, and therefore is an almost coboundary by Proposition \ref{prop:sublin}, using that $G$ has a controlled F\o lner sequence.
\end{proof}

\subsection{Dynamical criteria for Properties $\WAPT$ and $\WAPAP$}

We are now ready to state and prove the main result of this section. 

\begin{thm}\label{prop:main}
Let $G$ be locally compact group, and let $N$ be a closed subgroup. Let $H$ be the closure $\overline{MN}$ of the subgroup generated by $N\cup M$, where $M$ is the normal subgroup generated by $\Contr(N)$. Assume:
\begin{enumerate}
\item\label{item1} $G$ has a controlled F\o lner sequence;
\item\label{item2} $G$ is $N$-combable;
\item\label{item3}  $N$ has Property $\WAPT$;
\item\label{item4} $H$ is dense in $G$.
\end{enumerate}
Then $G$ has Property $\WAPT$.

Still assume (\ref{item1}) and (\ref{item2}) along with: 
\begin{enumerate}[(1')]
  \setcounter{enumi}{2}
\item\label{item3b}  $N$ has Property $\WAPAP$;
\item\label{item4b} $H$ is cocompact in $G$.
\end{enumerate}
Then $G$ has Property $\WAPAP$.
\end{thm}

\begin{proof}
Suppose that (\ref{item1}), (\ref{item2}), (\ref{item3}), (\ref{item4}) hold. Let $V$ be a Banach $G$-module with $V^G=0$. By (\ref{item4}) and Lemma \ref{prop:C}(\ref{ifhg}), $V^N=0$. By (\ref{item3}), $\overline{H^1}(N,V)=0$. By Lemma \ref{cohinj} and (\ref{item1}), (\ref{item2}), it follows that $\overline{H^1}(G,V)=0$. So $G$ has Property $\WAPT$.

Similarly, suppose that (\ref{item1}), (\ref{item2}), (\ref{item3b}'), (\ref{item4b}') hold. Let $V$ be a Banach $G$-module with $V^{G,\mathrm{ap}}=0$. By (\ref{item4b}') and Lemma \ref{prop:C}(\ref{ifhco}), $V^{N,\mathrm{ap}}=0$. By (\ref{item3b}'), $\overline{H^1}(N,V)=0$. By Lemma \ref{cohinj} and (\ref{item1}), (\ref{item2}), it follows that $\overline{H^1}(G,V)=0$. So $G$ has Property $\WAPAP$.
\end{proof}

\section{Groups in the class $\mathfrak{C}$}

\subsection{Families of examples in the class $\mathfrak{C}$}
We start proving that two important classes of groups belong to the class $\mathfrak{C}$ introduced in Definition \ref{defc}.

\begin{prop}\label{ex_clasc}
The following groups belong to the class $\mathfrak{C}$:
\begin{enumerate}
\item\label{item_rt} real triangulable connected Lie groups;
\item\label{qpgr} groups of the form $G=\mathbb{G}(\mathbf{Q}_p)$, where $\mathbb{G}$ is a connected linear algebraic group defined over $\Q_p$ such that $G$ is compactly generated and amenable.
\end{enumerate}
\end{prop}
\begin{proof}
(\ref{item_rt}) Let $G$ be a real triangulable connected Lie group. Let $N$ be a Cartan subgroup and $U=V$ its exponential radical. So $U$ is the intersection of the lower central series, $N$ is nilpotent, and $G=UN$. The $N$-action on the Lie algebra $\g$ induces a grading on $\g$, valued in the dual $\Hom(\mk{n},\R)$ of the abelianization of $\mathfrak{n}$, for which $\mathfrak{n}=\g_0$ and $\mathfrak{u}$ is the subalgebra generated by the $\g_\alpha$ for $\alpha\neq 0$. By definition, a weight is an element $\alpha$ in $\Hom(\mk{n},\R)$ such that $\g_\alpha\neq\{0\}$. Then there exists $k$ and a (possibly non-injective) family of nonzero weights $\alpha_1,\dots,\alpha_k$ such that, denoting $V_i=V_{\alpha_i}$, we have $U=V_1\dots V_k$; we omit the details since this is already covered by \cite[\S 6.B]{CTDehn}. The existence of elements in $N$ acting as contraction on $V_i$ is straightforward, and thus all the data fulfill the requirements to belong to the class $\mk{C}$.

(\ref{qpgr}) Write $\mathbb{G}=\mathbb{U}\rtimes (\mathbb{D}\mathbb{K})$, where $\mathbb{U}$ is the unipotent radical, and some reductive Levi factor is split into its $\Q_p$-split part $\mathbb{D}$ (a torus, by the amenability assumption) and its $\Q_p$-anisotropic part $\mathbb{K}$ (so $\mathbb{K}(\Q_p)$ is compact). Define $N=\mathbb{D}(\Q_p)$ and $V=\mathbb{U}(\Q_p)$. Thus $U=V\rtimes\mathbb{K}(\Q_p)$, so $V$ is closed cocompact in $U$. Then in the grading on $\mathfrak{g}$ induced by the $D$-action (which takes values in a free abelian group of rank $\dim(D)$, namely the group of multiplicative characters of $D$), $\mathfrak{v}$ is generated by the nonzero weights: as a consequence of the assumption that $G$ is compactly generated. Then the proof can be continued as in the real case (being also covered by the work in \cite[\S 6.B]{CTDehn}).   
\end{proof}

\begin{prop}\label{vcalfd}
Every virtually connected amenable Lie group $G$ belongs to the class $\mathfrak{C}'$ (Definition \ref{defcp}).
\end{prop}
\begin{proof}
By \cite[Lemma 3.A.1]{CTDehn} (based on \cite[Lemmas 2.4 and 6.7]{cornulier}), there exist copci homomorphisms $G\leftarrow G_1\to G_2\leftarrow G_3$ with $G_3$ triangulable. In addition, $G_1\to G_2$ has normal image (as it is mentioned in the proof of that lemma that $G_2$ is generated by its center and the image of $G_1$). Thus $G$ belongs to the class $\mathfrak{C}'$.
\end{proof}

\subsection{Controlled F\o lner sequences for groups in the class $\mathfrak{C}"$}\label{ccfs}

\begin{defn}\label{strongf}
In a CGLC group $G$ with compact generating symmetric subset $S$ with 1 and left Haar measure $\mu$, we call strong controlled F\o lner sequence, a sequence of positive measure compact subsets $(F_n)$ such that $F_n$ belongs to a ball of radius $O(n)$ and 
$$\frac{\mu(F'_n)}{\mu(F_n)}=O(1),$$ 
where $F'_n=F_nS^n$ is the $n$-tubular neighborhood of $F_n$ with respect to word metric associated to $S$.
\end{defn}

In \cite{tes1}, the pair $(F_n,F'_n)$ is called a controlled F\o lner pair. 
An easy argument \cite[Proposition 4.9]{tes1} shows that if $(F_n)$ is a strong controlled F\o lner sequence, then there exists $k_n\in\{1,\dots,n\}$ such that $(F_nS^{k_n})$ is a controlled F\o lner sequence (Definition \ref{ctvcontro}).

We shall need the following result from \cite{CTDehn} (which is proved there for a smaller class of groups but the proof readily extends to our setting).

\begin{lem}\label{lem:BallC}\cite[Theorem 6.B.2]{CTDehn}
Fix a CGLC group $G=UN$ in the class $\mathfrak{C}"$. Let $S$ and $T$ be compact generating subsets of respectively $G$ and $N$. There exist constants  $\mu_1$ and $\mu_2>1$ such that the following holds.
For every small enough norm $\|\cdot\|$ on $\mk{u}$, 
denoting by $U[r]$ the exponential of the $r$-ball in $(\mk{u},\|\cdot\|)$, we have, for all $n$, the inclusions
\[S^n\subset U[\mu_2^n]T^n,\quad   U[\mu_1^n]T^n\subset S^{2n}.\]
\end{lem}

By a straightforward application of the Baker-Campbell-Hausdorff formula, we obtain the following lemma.
\begin{lem}\label{lem:nonarchGroupBall}
Let $\mk{u}$ be a finite-dimensional nilpotent Lie algebra over a finite product of non-archimedean local fields of characteristic zero.  
Consider a norm $\|\cdot\|$ on $\mk{u}$. There exists $d\in \N$ such that  for all $r\geq 2$
$$\langle U[r]\rangle \subset U[r^d],$$ 
where $\langle U[r]\rangle$ is the (compact) subgroup generated by $U[r].$
\end{lem}

In the proof of \cite[Theorem II.1]{Guih}, Guivarc'h provides an asymptotic description of $K^n$, where $K$ is a compact symmetric generating subset of a nilpotent connected Lie group $G$, which in particular implies the following lemma.
\begin{lem}\label{lem:realGroupBall}
Let $\mk{u}$ be a finite-dimensional nilpotent Lie algebra over $\R$.  
Consider a norm $\|\cdot\|$ on $\mk{u}$. For every compact symmetric generating subset $K$ of $U$, there exists $C>1$ and $\eps>0$ such that  for all integers  $r\geq 2$,
$$U[\eps r^{1/C}]\subset K^r\subset U[r^C].$$ 
\end{lem}

We deduce the following corollaries:
\begin{cor}\label{cor:BallCnonArch}
Under the assumptions of Lemma \ref{lem:BallC}, and assuming in addition that 
$U$ is totally disconnected, there exist constants  $\lambda_1$ and $\lambda_2>1$ such that the following holds.
For every norm $\|\cdot\|$ on $\mk{u}$, 
denoting by $U[r]$ the exponential of the $r$-ball in $(\mk{u},\|\cdot\|)$, we have, for all large enough $n$, the inclusions
\[S^n\subset \langle U[\lambda_2^n]\rangle T^n,\quad   \langle U[\lambda_1^n]\rangle T^n\subset S^{2n}.\]
\end{cor}

\begin{cor}\label{cor:BallCReal}
We keep the assumptions of Lemma \ref{lem:BallC}, assuming in addition that $U$ is connected, and that $K$ is a compact generating set of $U$. Then there exist constants  $\beta_1$ and $\beta_2>1$ such that the following holds. For all large enough $n$, we have the inclusions
\[S^n\subset K^{\beta_1^n} T^n,\quad  K^{\beta_2^n}T^n\subset S^{2n}.\]
\end{cor}

Finally, we shall use 
\begin{lem}\label{lem:actionBall}
Under the assumptions of Lemma \ref{lem:BallC}, there exists $\lambda>1$ such that for all $r\geq 1$, $n\in \N$, $g\in U[r]$, and $h\in T^n$,
$$h^{-1}gh\in U[\lambda^nr].$$
In particular, if $U$ is connected, and $K$ is a compact generating subset of $U$, there exist $\alpha,b\geq 0$ such that for all  integers $r\geq 2$, $n\geq 1$, for all $g\in K^r$ and $h\in T^n$,
$$h^{-1}gh\in K^{\lceil \alpha^nr^{b}\rceil}.$$
\end{lem}
\begin{proof}
Let $\lambda$ be the supremum over all $t\in T$ of the operator norm of $t$ acting on the normed vector space $(\mk{u},\|\cdot\|)$. The first statement is now a direct consequence of the fact that the operator norm is submultiplicative. The second statement follows from Lemma \ref{lem:realGroupBall}. 
\end{proof}

\begin{thm}\label{th:C_strong}
Every CGLC group $G=NU$ in the class $\mathfrak{C}"$ (in particular in the class $\mathfrak{C}$) admits a strong controlled F\o lner sequence.
\end{thm}

\begin{proof} 
We write $U=U_{\R}\times U_{\mathrm{na}}$, where $U_{\R}$ is connected, and $U_{\mathrm{na}}$ is totally disconnected.
Fix some large enough integer $\mu$ (to be specified later) and define
$$F_n=(K^{\mu^n}\times \langle U_{\mathrm{na}}[\mu^n]\rangle)T^n,$$
where $K$ is a compact symmetric generating subset of $U_{\R}$.

 By Corollaries \ref{cor:BallCnonArch} and \ref{cor:BallCReal}, there exists $C>0$ such that $F_n\subset S^{Cn}$. 
Now observe that if $\mu$ is large enough, Lemma \ref{lem:actionBall}  implies that 
$$\langle U_{\mathrm{na}}[\mu^n]\rangle T^n\langle U_{\mathrm{na}}[\lambda_1^n]\rangle=\langle U_{\mathrm{na}}[\mu^n]\rangle T^{n}.$$
 On the other hand, assuming $\mu\geq \beta_1^b\lambda$, we have
 $$K^{\mu^n}T^nK^{\beta_1^n}\subset K^{\mu^n+\lambda^n\beta_1^{bn}}T^n\subset K^{2\mu^n}T^n.$$
We deduce that
 $$F_n'\subset (\langle U_{\mathrm{na}}[\mu^n]\rangle \times K^{2\mu^n}) T^{2n}.$$
 Finally, in order to conclude, we observe that by the doubling property for both $N$ and $U_{\R}$, there exists an integer $k$ such that for all all $n$ there exist finite subsets $X_n\subset N$ and $Y_n\subset U_{\R}$ of cardinality at most $k$ such that $T^{2n}\subset T^nX_n$ and $K^{2\mu^n}\subset Y_n K^{\mu^n}$. Hence we have
$$ F_n'\subset Y_nF_nX_n,$$ from which we deduce that
\[|F_n'|\leq k^2|F_n|.\qedhere\]
\end{proof}

\subsection{Combability of groups  in the class $\mathfrak{C}"$}\label{ccomb}

\begin{thm}\label{thm:Ncombable}
CGLC groups in the class $\mathfrak{C}"$ are $N$-combable in the sense of Definition \ref{def:combing}.
\end{thm}
\begin{proof}
Let $G\in \mathfrak{C}$ and let $S$ be a compact generating subset of $G$. For convenience, we shall assume that the generating set $S=S_U\cup S_N$, where $S_U\subset U$ (resp.\ $S_N\subset N$). We assume that $i=1$, namely that $G=UN$, with $U$ being a nilpotent algebraic group over some local field $\K$; the general case being similar. Let $q:G\to M:=G/U$. For every $g\in G$ of size $n$, its projection has length $\leq n$ with respect to $q(S)$, hence $g$ can be written as a product $g=um$, such that $m$ has length equal to $|q(g)|\leq n$, and $u$ has length $\leq n+|q(g)|\leq 2n$. Therefore, it is enough to prove (\ref{eq:combing})  for $g=u$.

Consider a finite-dimensional faithful representation of $U$ as unipotent matrices in  $M_d(\K)$ and equip the latter with a submultiplicative  norm $\|\cdot\|$. 
 We shall use the notation $\preceq$ and $ \simeq$ to mean ``up to multiplicative and additive constants".

Moreover, an easy calculation (using that $U$ is unipotent) shows that for all $u_1,\ldots, u_n \in U$, 
\begin{equation}\label{eq:unipotentnorm}\|u_1\ldots u_n\|\preceq n^d\max_i\|u_i\|.\end{equation}
We shall also use the fact that given a norm $\|\cdot\|_{\mathrm{Lie}}$ on the Lie algebra $\mathfrak{u}$ of $U$, one has 
$$\log \|u\|\simeq \log \|\log(u)\|_{\mathrm{Lie}},$$
where $\log: U\to \mathfrak{u}$ is the inverse of the exponential map. This estimate follows from the Baker-Campbell-Hausdorff formula, using that  $\log$ and $\exp$ are  polynomial maps.
The action by conjugation of $N$ on $U$ induces a group homomorphism $N\to \Aut(\mathfrak{u})$. Let $\|\cdot\|_{\mathrm{op}}$ be the operator norm on  $\End(\mathfrak{u})$, that by abuse of notation we consider as a norm on elements of $N$. Let $C=\max_{m \in S_N}\|m\|_{\mathrm{op}}$ and $K=\max_{u\in S_U} \|u\|_{\mathrm{Lie}}$. 

\begin{lem}\label{lem:word/norm}
For all $u\in U$, $|u|_S\succeq \log \|u\|$.   
\end{lem}

\begin{proof}
Assume that $|u|_S=n$, and so that $u=m_1u_1\ldots m_nu_n$, where the $u_i\in S_U$ and $m_i\in S_N$. Denote $h^g=g^{-1}hg$. One has the following formula 
$$u=u_1^{m_1}u_2^{m_1m_2}\ldots u_n^{m_1\ldots m_n}.$$
by (\ref {eq:unipotentnorm}),
by submultiplicativity, one has
$$\|u\|\preceq n^d\max_i\|u_i^{m_1\ldots m_i}\|.$$
On the other hand, for every $i$,  $\|u^{m_1\ldots m_i}\|_{\mathrm{Lie}}\leq \|m_1\ldots m_i\|_{\mathrm{op}}\|u_i\|_{\mathrm{Lie}}\leq KC^n$. The lemma follows.  
\end{proof}

In addition, we have
\begin{lem}\cite[Lemma 6.B.3.]{CTDehn}
There exists an integer $\ell$ such that every element $x\in U$ can be written as $v_1\ldots v_{\ell}$ with $v_i\in \bigcup_jV_j$, and $\max \|v_i\|\preceq \|u\|.$
\end{lem}
Thanks to  the previous lemma, it is enough to treat the case where $U=V_j$. Therefore we can assume that there exists $t\in N$ contracting all of $U$. Up to replacing it by some power, we can assume that $\|t\|_{\mathrm{op}}\leq 1/2$. For convenience, let us assume that $S_U$ contains all elements $u\in U$ such that $\|u\|_{\mathrm{Lie}}\leq 1$.
Given an element $u\in U$ such that $\|u\|_{\mathrm{Lie}}\leq 2^n$, it follows that the element $u^{t^n}$ belongs to $S_U$. It follows from Lemma \ref{lem:word/norm}, that every element $u\in U$ such that $|u|_S\leq n$ can be written as $t^{n'}u_0t^{-n'}$, where $n'\preceq n$ and $u_0\in S_U$. This finishes the proof that $G$ is $N$-combable. 
\end{proof}

\section{Proof of Theorem \ref{Main} and other results}
\subsection{Proof of Theorem \ref{Main}}
We need to check all four requirements of Theorem \ref{prop:main}:
\begin{itemize}
\item $G$ has a controlled F\o lner sequence: this is done in \S\ref{ccfs}.
\item $G$ is $N$-combable: this is done in \S\ref{ccomb}.
\item $N$ has Property $\WAPT$: this is easy (Corollary \ref{nilwapt} with Proposition \ref{prop:compactExt})
\end{itemize}

The last requirement, (\ref{item4}), can actually fail, but we can arrange it to hold enlarging $N$, replacing it with $N'=NW$ for some suitable compact subgroup $W$ normalized by $N$. Namely, we use:

\begin{lem}
Let $U$ be a open subgroup in a finite product $L$ of unipotent real and $p$-adic fields. Then the divisible subgroup $U_{\mathrm{div}}$ of $U$ is closed, cocompact in $U$ and contains the real component.
\end{lem}
\begin{proof}
Since $U$ decomposes as a product over the components, we can suppose $L$ is real or $p$-adic for a single $p$. In the real case, necessarily $U=L$. Suppose that $L$ is $p$-adic. Then $U_{\mathrm{div}}=\bigcap_n \phi^n(U)$, where $\phi(u)=u^p$. Since $u$ is a self-homeomorphism of $L$, $\phi^n(U)$ is closed and hence $U_{\mathrm{div}}$ is closed; this implies that it is a $p$-adic subgroup, and it easily follows that it is Zariski-closed. An extension of two divisible nilpotent groups is divisible. In particular, in the quotient $L/U_{\mathrm{div}}$, the open subgroup group $U/U_{\mathrm{div}}$ has no $p$-divisible element. Since $\phi$ contracts to 0, we can deduce that $U/U_{\mathrm{div}}$ is compact. 
\end{proof}

\begin{lem}
Let $G$ be a group in the class $\mathfrak{C}$, with $U,N$ as in the definition. Then $G$ has a compact subgroup $W\subset U$, normalized by $N$, such that $U=U_{\mathrm{\mathrm{div}}}W$.
\end{lem}
\begin{proof}
The quotient $U/U_{\mathrm{\mathrm{div}}}$ is compact, and is a product of various unipotent $p$-adic groups. So the $N$-action is necessarily distal (only eigenvalues of modulus 1). 

Let $U_{\mathrm{na}}$ be the elliptic radical of $U$, so $U_{\mathrm{na}}\times U^\circ$. Let $U_1$ be the distal part of the $N$-action on $U_{\mathrm{na}}$. Then the restriction to $U_{\mathrm{na}}$ of the quotient map $U\to U/U_{\mathrm{\mathrm{div}}}$ is surjective. Moreover, $U_1$ is an increasing union of its $N$-invariant compact open subgroups. Hence there exists an $N$-invariant compact open subgroup $W$ of $U_1$ whose image in $U/U_{\mathrm{\mathrm{div}}}$ is surjective.
\end{proof}

To conclude, we define $N'=NW$. This does not affect the first condition, nor the second since we pass to a larger subgroup $N$. Since Property $\WAPT$ is also invariant under extensions by compact kernels (Proposition \ref{prop:compactExt}), $N'$ has Property $\WAPT$. Finally, the last verification (Lemma \ref{divcon}) is that the subgroup generated by $\mathrm{Contr}(N')$ is equal to $U_{\mathrm{div}}$ and we deduce that $H=G$.

\begin{lem}\label{divcon}
The subgroup $V$ generated by $\mathrm{Contr}(N)$ is equal to $U_{\mathrm{\mathrm{div}}}$.
\end{lem}
\begin{proof}
Clearly, the image of $\mathrm{Contr}(N)$ in the compact group $U/U_{\mathrm{\mathrm{div}}}$ is trivial. Hence $\mathrm{Contr}(N)\subset U_{\mathrm{\mathrm{div}}}$. Thus $\bar{V}$ is cocompact in $U_{\mathrm{\mathrm{div}}}$. It is easy to see that $U_{\mathrm{\mathrm{div}}}$ has no proper cocompact subgroup: indeed, if $U_{\mathrm{\mathrm{div}}}$ is abelian, its quotient by $\bar{V}$ is a compact, divisible totally disconnected abelian group and apart from the trivial group, this does not exist (since nontrivial profinite groups have nontrivial finite quotients). So, when $U_{\mathrm{\mathrm{div}}}$ is abelian, we deduce that $V$ is dense. But in this case, it is clear that $V$ is closed, since it is generated by some eigenspaces. So $V=U_{\mathrm{\mathrm{div}}}$ when $U$ is abelian.

In general, we deduce that $V[U_{\mathrm{\mathrm{div}}},U_{\mathrm{\mathrm{div}}}]=U_{\mathrm{\mathrm{div}}}$, that is, $V$ generates the nilpotent group $U_{\mathrm{\mathrm{div}}}$ modulo commutators, and deduce that $V=U_{\mathrm{\mathrm{div}}}$.
\end{proof}

\subsection{Other results}

\begin{thm}\label{cccap}
Every compactly generated locally compact group $G$ having an open subgroup $G'$ of finite index in the class $\mathfrak{C}''$ (Definition \ref{dcpp}) has Property $\WAPFD$.
\end{thm}
\begin{proof}
Along with Theorem \ref{polgr} and using the second part of Theorem \ref{prop:main} the above proof shows that $G'$ has Property $\WAPAP$. That $G'$ has Property $\WAPAP$ follows from Proposition \ref{prop:subgroup}(\ref{ps1}).

Then we observe that $G/\KT(G)$ has polynomial growth (see \S\ref{pAPFD} for the definition of $\KT(G)$), by Lemma \ref{gkt_poly}, and hence has Property $\WAPFD$ by Theorem \ref{polgr}. Hence $G/\KT(G)$ has Property $\APFD$, and in turn, by Proposition \ref{apfd_kt}, $G$ has Property $\APFD$. Since $G$ has Property $\WAPAP$, this shows that $G$ has Property $\WAPFD$.
\end{proof}

Note that by Proposition \ref{apt}, we have a criterion whether $G$ has Property $\WAPT$, namely if and only if the group of polynomial growth $G/\KT(G)$ has Property $\WAPT$.

\begin{proof}[Proof of Corollary \ref{CPWAPAP}]
The statement is that locally compact groups in the class $\mathfrak{C}'$ (Definition \ref{defcp}) have Property $\WAPFD$. Indeed, consider $G\to G_1\leftarrow G_2\to G_3$ as in Definition \ref{defcp}. Since $G_3$ belongs to the class $\mathfrak{C}$, it has Property $\WAPT$ by Theorem \ref{Main} Since $G_2\to G_3$ is copci and $G_3$ has Property $\WAPT$, $G_2$ has Property $\WAPT$ by Theorem \ref{thm:subgroup} and Proposition \ref{prop:compactExt}. Since $G_2\to G_1$ is copci with normal image, it follows that $G_1$ has Property $\WAPFD$ by Proposition \ref{prop:subgroup} and Proposition \ref{prop:compactExt}. By Theorem \ref{thm:subgroup} again (for Property $\WAPFD$ this time) and Proposition \ref{prop:compactExt}, it follows that $G$ has Property $\WAPFD$.
\end{proof}

\begin{proof}[Proof of Corollary \ref{vcalfd2}]
This follows from Corollary \ref{CPWAPAP} in combination with Proposition \ref{vcalfd}. 
\end{proof}

\begin{thm}\label{cor:Delorme}
Let $G$ be a connected solvable Lie group. Then every WAP Banach $G$-module with nonzero reduced first cohomology has a 1-dimensional factor (with nonzero first cohomology).
\end{thm}
\begin{proof}
This follows from Corollary \ref{vcalfd2}, using that finite-dimensional unitary irreducible representations of connected solvable Lie groups have complex dimension one.
\end{proof}

\section{Subgroups of $\GL(n,\Q)$}

\subsection{Unipotent closure}
Let $(G_i)$ be a family of locally compact groups, with given compact open subgroups $K_i$. The corresponding semirestricted product is the subgroup of $\prod_i G_i$ consisting of families of whose coordinates are in $K_i$ with finitely many exceptions (in other words, it is the subgroup generated by $\prod K_i$ and $\bigoplus G_i$); it has a natural group topology for which $\prod K_i$ is a compact open subgroup. We denote it by $\prod_i^{(K_i)}G_i$.

If $H_i\subset G_i$ is a family of closed subgroups, $\prod_i^{(K_i\cap H_i)}H_i$ naturally occurs as a closed subgroup of $\prod_i^{(K_i)}G_i$. We call it a standard subgroup (according to this given decomposition).

Now assume that, $p$ ranging over the prime numbers, $G_p$ is a $p$-elliptic locally compact group (in the sense that every compact subset of $G$ is contained in a pro-$p$-subgroup). Then we have

\begin{lem}\label{clostan}
Every closed subgroup $H$ of $\prod^{(K_p)}G_p$ is standard. 
\end{lem}
\begin{proof}
Let us show that $H$ is closed under taking under all projections. Fix a prime $q$. That $\Z$ is dense in $\prod_p\Z_p$ implies that there exists a sequence $(n_i)$ in $\Z$ such that $n_i\to 1$ in $\Z_q$ and $n_i\to 0$ in $\Z_p$ for all $p\neq q$. Then for every $x\in \prod^{(K_p)}G_p$, the sequence $x^{n_i}$ tends to the projection of $x$ on $G_p$. In particular, if $x\in H$, then this projection also belongs to $H$.

Now let $H_p$ be the projection of $H$ on $G_p$. Then the closed subgroup generated by the $H_p$ contains $\prod_p (H_p\cap K_p)$ and contains $\bigoplus_p H_p$. Thus $\prod^{(K_p\cap H_p)}H_p$ is contained in $H$. Conversely, $H$ is contained in both $\prod_p H_p$ and $\prod^{(K_p)}G_p$, and the intersection of these two is by definition $\prod^{(K_p\cap H_p)}H_p$, so $H$ is contained in $\prod^{(K_p\cap H_p)}H_p$; we conclude that these subgroups are equal.
\end{proof}

Recall that the ring of {\it finite adeles} is the semirestricted product $\mathbf{A}=\prod_p^{(\Z_p)}\Q_p$; the diagonal inclusion embeds $\mathbf{Q}$ as a dense subring into $\mathbf{A}$ and as a discrete cocompact subring in $\mathbf{A}\times\R$ (the latter is known as ring of {\it adeles}).

\begin{defn}Let $H$ be a subgroup of $\GL_m(\Q)$. We define its fine closure as the subgroup $\mathcal{C}(H)$ of $\GL_m(\mathbf{A}\times\R)$ generated by the closures of the various $p$-adic projections $\pi_p(H)$, and the Zariski closure $\mathcal{C}_0(H)$ of the real projection.
\end{defn}

\begin{lem}\label{coficlo}
Let $H$ be a subgroup of $\GL_m(\Q)$ conjugate to a subgroup of upper unipotent matrices. Then $H$ is cocompact in its fine closure $\mathcal{C}(H)$. 
\end{lem}

(It is well-known that the condition is equivalent to assuming that each element of $H$ is unipotent.)

\begin{proof}
Denote by $\mathcal{C}_p(H)$ the closure of the projection of $H$ in $\GL_m(\Q_p)$ and $\mathcal{C}_+(H)$ the closure of subgroup they generate in $\GL_m(\mathbf{A})$, so that $\mathcal{C}(H)=\mathcal{C}_+(H)\times \mathcal{C}_0(H)$. Denote by $\pi_+$ and $\pi$ the natural embeddings $\GL_m(\Q)\to\GL_m(\mathbf{A})$ and $\GL_m(\Q)\to\GL_m(\mathbf{A}\times\R)$.

By the assumption, $\mathcal{C}_p(H)$ is $p$-elliptic for every prime $p$. Then by Lemma \ref{clostan}, $\pi_+(H)$ is dense in $\mathcal{C}_+(H)=\prod_p^{(\mathcal{C}_p(H)\cap\Z_p)}\mathcal{C}_p(H)$. Let $K$ be the compact open subgroup $\prod_p(\mathcal{C}_p(H)\cap\Z_p)$. Then this implies that $\mathcal{C}_+(H)=\pi_+(H)K$. Therefore, $\mathcal{C}(H)=\pi(H)(K\times \mathcal{C}_0(H))$.

We claim that the projection of $\pi(H)\cap (K\times \mathcal{C}_0(H))$ is cocompact in $\mathcal{C}_0(H)$. Indeed, in a connected unipotent real group $U$, a subgroup is cocompact if and only if it is Zariski-dense, if and only it is not contained in the kernel any nonzero homomorphism $U\to\R$. Assume by contradiction we have such a homomorphism $f$ on $\mathcal{C}_0(H)$. Pick $(u,b)\in\pi(H)$ with $u\in \mathcal{C}_+(H)$ and $b\in \mathcal{C}_0(H)$ with $f(b)\neq 0$. Then there exists $n\ge 1$ such that $u^n\in K$. Hence $(u^n,b^n)\in\pi(H)\cap (K\times \mathcal{C}_0(H))$ but $f(b^n)=nf(b)\neq 0$, a contradiction. 

So the projection of $\pi(H)\cap (K\times \mathcal{C}_0(H))$ is cocompact in $\mathcal{C}_0(H)$. This implies (pulling back by a quotient homomorphism with compact kernel) that $\pi(H)\cap (K\times \mathcal{C}_0(H))$ is cocompact in $K\times \mathcal{C}_0(H)$. Since $K\times \mathcal{C}_0(H)$ is an open subgroup and $\mathcal{C}(H)=\pi(H)(K\times \mathcal{C}_0(H))$, this implies that $\pi(H)$ is cocompact in $\mathcal{C}(H)$.
\end{proof}

Now, let $\Gamma$ be a finitely generated, virtually solvable subgroup of $\GL_m(\Q)$. Let $U$ be its unipotent radical (the intersection with the unipotent radical of its Zariski closure, which is also the largest normal subgroup of $\Gamma$ consisting of unipotent elements). So $\Gamma/U$ is finitely generated and virtually abelian. Identify $\Gamma$ with its image in $\GL_m(\mathbf{A}\times\R)$.

\begin{prop}\label{unipclo}
$\mathcal{C}(U)$ is open in $\Gamma\mathcal{C}(U)$, which is closed in $\GL_m(\mathbf{A}\times\R)$, and $\Gamma$ is cocompact in $\Gamma\mathcal{C}(U)$.
\end{prop}
\begin{proof}
Choose a partial flag in $\Q^m$ that is $\Gamma$-invariant with irreducible successive quotients. This yields an upper block-triangular decomposition. Then $U$ acts trivially on the irreducible subquotients, which means it acts by matrices with identity diagonal blocks. Let $\phi$ be the homomorphism mapping a matrix that is upper triangular in this decomposition to its ``diagonal trace", that is, replacing all upper unipotent blocks with 0. Then the kernel of $\phi:\Gamma\to\phi(\Gamma)$ is exactly $U$. Moreover, $\phi$ extends to $\Gamma\mathcal{C}(U)$ and $\phi(\Gamma\mathcal{C}(U))=\phi(\Gamma)$, which is discrete. Hence the kernel $\mathcal{C}(U)$ of $\phi:\Gamma\mathcal{C}(U)\to\phi(\Gamma)$ is open in $\Gamma\mathcal{C}(U)$. This implies in particular that $\Gamma\mathcal{C}(U)$ is closed in $\GL_m(\mathbf{A}\times\R)$. 

Lemma \ref{coficlo} ensures that $U$ is cocompact in $\mathcal{C}(U)$, and the cocompactness statement follows.
\end{proof}

\begin{rem}
A related embedding construction is performed by Shalom and Willis in the proof of \cite[Prop.\ 7.5]{ShW13}, in the context of certain lattices in semisimple groups.
\end{rem}

\subsection{Partial splittings}

\begin{lem}\label{gs1}
Let $M$ be a locally compact group and $s$ a contracting automorphism of $M$. Define $f(g)=gs(g)^{-1}$. Then $f$ is a self-homeomorphism of $M$.
\end{lem}
\begin{proof}
We wish to define its inverse as $F(g)=gs(g)s^2(g)\dots$; we need to check that this product is ``summable" (uniformly on compact subsets, namely that $\prod_{k=n}^{n+\ell}s^k(g)$ tends to 1 when $n$ tends to $+\infty$, uniformly in $\ell$ and for $g$ in any given compact subset. 

If $M$ is totally disconnected, then it has a compact open subgroup $K$ such that $s(K)\subset K$, and then $\bigcap_{n\ge 0}s^n(K)=\{1\}$. Then the summability condition immediately follows.

If $M$ is connected, then $M$ is a finite-dimensional real vector space with a linear contraction and the summability is a standard verification.

The general case follows from the fact that $M$ decomposes canonically as topological direct product of $M^\circ$ and its elliptic radical, which is totally disconnected \cite[Prop.\ 4.2]{Sie}.

Once the summability is established, it is immediate that $F\circ f$ and $f\circ F$ are both the identity of $G$.
\end{proof}

\begin{lem}
Let $G$ be a locally compact group in an extension $1\to M\to G\to A\to 1$, with $M$ and $A$ abelian. Assume that some element $g$ of $G$ right-acts on $M$ as a contraction. Then the centralizer of $g$ is a section of the extension.
\end{lem}
\begin{proof}
Let $H$ be the centralizer of $g$. Clearly $H\cap M=\{1\}$. So it is enough to show that $HM=G$. Equivalently, letting $h$ be any element of $G$, we have to show that the equation $[hm,g]=1$ has a solution $m\in M$. Here the commutator is defined as $[X,Y]=X^{-1}Y^{-1}XY$ and satisfies the identity $[XY,Z]=[X,Z]^Y.[Y,Z]$. Then $[hm,g]=[h,g][m,g]$. By Lemma \ref{gs1}, $m\mapsto [m,g]$ is a self-homeomorphism of $M$. So indeed we obtain a unique solution.
\end{proof}

In the following lemma, we refer to Definition \ref{contr} for the definition of $\mathrm{Contr}(F)$.

\begin{lem}\label{existsplit}
Let $G$ be a locally compact group in an extension $1\to U\to G\to A\to 1$, with $A$ compactly generated abelian and $U$ sub-unipotent (over a finite product of adic and real fields), i.e., a closed subgroup of a unipotent group containing the real component. Then $G$ has a compactly generated, closed subgroup $F$ such that
 $FU=G$ and $F$ has polynomial growth, and $F\cap U^\circ$ is the distal part of $U^\circ$. Moreover, if $G$ is compactly generated, the subgroup generated by $F$ and $\mathrm{Contr}(F)$ is cocompact.
\end{lem}
\begin{proof}
We first prove the result with $F$ not assumed compactly generated (so polynomial growth means that all its open, compactly generated subgroups have polynomial growth). The result immediately follows, since we can replace then $F$ with a large enough compactly generated open subgroup $F'$ still satisfying $F'U=G$ (since $G/U$ is compactly generated).

We argue by induction on the dimension of $U$ (the sum of its real and $p$-adic dimensions for various $p$). If $\dim(U)=0$, then $U=1$ and the result is trivial.
Let $Z$ be the last term of the lower central series of $U$. Then $Z$ has positive dimension. If the action of $A$ on $Z$ is distal (i.e., all eigenvalues have modulus 1), we define $M=Z$; otherwise, in the Lie algebra we find an irreducible non-distal submodule, which corresponds to an irreducible submodule $M$ of $Z$; in both cases $M$ has positive dimension. We can argue by induction for $1\to U/M\to G/M\to A\to 1$ to get a closed subgroup $L/M$ of polynomial growth with $LN=G$ and $(L/M)\cap (U/M)^\circ$ is the distal part of $(U/M)^\circ$. If $Z$ is distal as $A$-module, then $L$ also has polynomial growth and we are done with $F=L$. Otherwise, $M$ is irreducible non-distal and $L$ contains the distal part of $U^\circ$. We have the extension $1\to L\cap U\to L\to L/L\cap U\to  1$. Note that $L\cap U$ is sub-unipotent, making use that $L\cap U^\circ$ is connected. If $\dim(L)<\dim(U)$, we can argue once more by induction within $L$ to find the desired subgroup. Otherwise, $L\cap U$ is open in $U$. This means that $(L\cap U)/M$ is open in $U/M$, and if this happens, $G/M$ has polynomial growth. If $U$ is not abelian, this forces $G$ to have  polynomial growth, and then we are done with $F=L$. Otherwise $U$ is abelian. In this case, if the distal part is nontrivial, we can argue in the same way with $M=D$. If the distal part is trivial, and we choose $M$ irreducible as above, the previous argument works as soon as $U/M$ is nontrivial. So the remaining case is when $U$ is irreducible and non-distal; in particular the distal part of $U^\circ$ is trivial. In this case, there exists an element acting as a contraction (if $g$ acts non-distally, the contraction part of either $g$ or $g^{-1}$ is a nonzero submodule, hence is all of $M$), so we can invoke Lemma \ref{existsplit}.

Now suppose that $G$ is compactly generated and let us prove the last statement. It is enough to show that the subgroup of $U$ generated by $(F\cap U)\cup\mathrm{Contr}(F)$ is cocompact, i.e., contains $U_{\mathrm{div}}$. First, it contains $U^\circ$, because we have ensured that $F\cap U$ contains the distal part. For the non-Archimedean part, that $G$ is compactly generated implies that $U$ is compactly generated as normal subgroup (because $G/U$ is compactly presented, being abelian), and hence it follows that the non-Archimedean part of $U_{\mathrm{div}}$ is contained in the subgroup generated by $\mathrm{Contr}(F)$ (indeed, otherwise we would obtain an $A$-equivariant quotient of $U$ isomorphic to $\Q_p^k$ with an irreducible distal action, for some $p,k$ and get a contradiction).
\end{proof}

\begin{prop}\label{propgp}
Let $\Gamma$ be a finitely generated amenable subgroup of $\GL_m(\Q)$, and let $G=\Gamma V\subset \GL_m(\mathbf{A}\times\R)$ with $V=\mathcal{C}(U)$ as defined as before Proposition \ref{unipclo}. Then $G$ has an open normal finite index subgroup $G'$ of the form $VF$ with $F$ compactly generated of polynomial growth, such that the subgroup of $G'$ generated by $F\cup\mathrm{Contr}(F)$ is cocompact. In particular, $G'$ belongs to the class $\mathfrak{C}''$ (Definition \ref{dcpp}).
\end{prop}
\begin{proof}
$\Gamma$ has a finite index subgroup $\Lambda$ whose Zariski closure is unipotent-by-abelian. It follows that $G'=\Lambda V$ is open normal of finite index in $G$, and $G'/V$ is abelian. So Lemma \ref{existsplit} applies. The last statement immediately follows.
\end{proof}

\begin{cor}\label{ccpp}
Every finitely generated amenable subgroup of $\GL_m(\Q)$ embeds as a cocompact lattice into a locally compact group $G$ with an open subgroup of finite index $G'$ in the class $\mathfrak{C}''$.\qed
\end{cor}

\begin{cor}\label{amglmq}
Every finitely generated amenable subgroup of $\GL_m(\Q)$ has Property $\WAPFD$, and hence has Property $\HFD$.
\end{cor}
\begin{proof}
Use the notation of Corollary \ref{ccpp}, so that $G'$ belongs to the class $\mathfrak{C}''$. By Theorem \ref{cccap}, we deduce that $G$ has Property $\WAPFD$. Hence by Theorem \ref{thm:subgroup}, $\Gamma$ has Property $\WAPFD$:
\end{proof}

\begin{rem}
Corollary \ref{amglmq} works when $\Q$ is replaced with any number field $K$, since $\GL_m(K)$ embeds into $\GL_{m[K:\Q]}(\Q)$.
\end{rem}

Using that every finitely generated VSP group is quotient of a virtually torsion-free finitely generated VSP group \cite{KL17}, and the fact that $\WAPFD$ passes to quotients, this can be improved to
\begin{cor}\label{cor:VSP}
Every finitely generated amenable VSP group has Property $\WAPFD$.
\end{cor}

\begin{cor}\label{cor:Jb}
For every finitely generated VSP group $G$ equipped with a finite generating subset $S$, there exists $c>0$ such that the $L^p$-isoperimetric profile inside balls (see \S\ref{ichas}) satisfies

\begin{equation}\label{eq:Jb} J^b_{G,p}(n)\geq cn.
\end{equation}
\end{cor}
\begin{proof}
By Lemma \ref{lem:BallC}, groups of the class $\mathfrak{C}''$ have a strong controlled F\o lner sequence (called a controlled F\o lner pair in \cite{tes1}). By \cite[Proposition 4.9]{tes1}, this implies that groups of the class $\mathfrak{C}''$ satisfy (\ref{eq:Jb}). On the other hand, $G$ being quasi-isometric to its cocompact hull, we deduce from Theorem \ref{thm:ccpp} and \cite[Theorem 1]{tes3} that every amenable finitely generated subgroup of $\GL_m(\Q)$ satisfies (\ref{eq:Jb}). Now, we once again apply the main result of Kropholler and Lorensen \cite{KL17} and the fact that $J^b_{\ast,p}$ behaves well under taking quotients \cite[Theorem 1]{tes2} to conclude. \end{proof}

\section{Mean ergodic theorem and Bourgain's theorem}\label{aux}
\subsection{Proof of Proposition \ref{ergoprop}}

The ``if" part was already addressed in the introduction. For the other direction, suppose that $G$ does not have $\WAPT$, but has $\WAPFD$. This means that $G$ has a finite-dimensional orthogonal representation $\pi$ with nonzero first reduced cohomology and no non-zero invariant vector. Now, recall that by the standard Gaussian construction (see \cite[Corollary A.7.15]{BHV08}), one can assume that $\pi$ is a subrepresentation of some orthogonal representation $\pi'$ of $G$ coming from a measure-preserving ergodic action on some probability space $X$. Let $b$ be a 1-cocycle for $\pi$ that is not an almost coboundary, and let $b'$ be the corresponding cocycle for $\pi'$: note that since $\pi$ has no invariant vectors, $b'$ is orthogonal to the space of constant functions. 
One therefore has $b'(g)(x)=c(g)(x)$ where $c$ is a square-integrable cocycle $c:X\times G\to \R$ of zero average. Using that $b'$ is not an almost coboundary, one easily checks that $\frac{1}{|g|}(c(g)(x))$ does not tend to zero in $L^2$-norm as $|g|\to \infty$. In particular the ergodic theorem for $G$ in $L^2$ (defined similarly) fails; since the inclusion $L^2(X)\to L^1(X)$ is continuous, it also fails in $L^1$.

\begin{rem}
The above proof works with no change if $G$ has Property $\HFD$ and not $\HT$. This is actually more general, since the reader can easily check that if $G$ has Property $\WAPFD$ but not $\WAPT$, then it has Property $\HFD$ but not $\HT$.
\end{rem}

\subsection{Proof of Corollary \ref{corbo}}

For every commutative unital ring $R$ and $t\in R^\times$, consider the group 
 $$A(R,t)=\left\{\left(\begin{array}{cc}
t^n & x\\
0 & t^{-n}
\end{array}\right); x\in R, n\in \Z \right\}\subset\GL_2(R).$$
Let us fix some prime $p$.
 Note that the ring $\F_p[t,t^{-1}]$ embeds densely in $\F_p(\!(t)\!)$, but the diagonal embedding  $\F_p[t,t^{-1}]\to \F_p(\!(t)\!)\oplus\F_p(\!(t)\!)$ sending $t$ to $(t,t^{-1})$  is easily seen to be discrete and cocompact. The lamplighter group 
$L_p=\mathbf{F}_p\wr \Z$ can be described as $A(\mathbf{F}_p[t,t^{-1}],t)$, and therefore embeds as a cocompact lattice in $G=(\mathbf{F}_p\lp t\rp)^2 \rtimes \Z$, where $\Z$ acts by multiplication by $t$ on the first factor and by $t^{-1}$ on the second factor. First, we note that this implies that $L_p$ has $\WAPT$. 
On the other hand the group $G$, and therefore $L_p$ quasi-isometrically embeds as a subgroup of $\big(A(\mathbf{F}_p\lp t\rp,t)\big)^2$.  Observe that $A(\F_p(\!(t)\!),t)$ acts properly and cocompactly on the Bass-Serre tree of $\SL(2,\F_p(\!(t)\!))$. 
It follows that $L_p$ embeds quasi-isometrically into a product of two $(p+1)$-regular trees. Therefore, in order to show Bourgain's theorem, it is enough to prove that $L_p$ does not quasi-isometrically embed into any superreflexive Banach space. Since $L_p$ is amenable, by \cite[Theorem 9.1]{NP} it is enough to prove that $L_p$ does not admit any affine isometric  action on some superreflexive Banach space $E$ whose orbits are quasi-isometrically embedded. Consider the $1$-cocycle $b$ associated to such an action. Since $L_p$ has Property $\WAPT$, this cocycle decomposes as $b=b_1+b_2$, where $b_1$ is a group homomorphism to $E$, and $b_2$ is an almost coboundary. Approximating $b_2$ by coboundaries, one easily checks that it is sublinear, namely $\|b_2(g)\|/|g|\to 0$ as $|g|\to \infty$ (where $|\cdot|$ is some arbitrary word metric on $L_p$). This clearly implies that $b$ cannot be a quasi-isometric embedding.  So the corollary is proved.

\begin{rem}Note that in Corollary \ref{corbo}, we only recover the qualitative part of Bourgain's theorem. Indeed, the latter also provides optimal quantitative estimates on the distortion (as in \cite{tes1}), which do not follow from the approach here.
\end{rem}


\baselineskip=16pt


\bigskip

\footnotesize

\end{document}